%% file: main.tex
\documentclass{amsart}
\usepackage[english]{babel}
\usepackage[top=1.5in,bottom=1.5in,left=1.6in,right=1.6in]{geometry}
\usepackage{euscript}
\usepackage{wallpaper}
\usepackage{amsmath}
\usepackage{amsthm}
\usepackage{amsfonts}
\usepackage{amssymb}
\usepackage{graphicx}
\usepackage{float}
\usepackage{subfigure}
\usepackage{makecell}
\usepackage{hyperref}
\usepackage{indentfirst}
\usepackage{enumerate}
\usepackage{mathrsfs}
\usepackage{verbatim}
\usepackage[normalem]{ulem}
\newcommand\redout{\bgroup\markoverwith{\textcolor{red}{\rule[0.5ex]{2pt}{0.8pt}}}\ULon}

\newcommand\remove[1]{}

\newtheorem{definitionenv}{Definition}[section]
\newtheorem{lemmaenv}[definitionenv]{Lemma}
\newtheorem{theoremenv}[definitionenv]{Theorem}
\newtheorem{corollaryenv}[definitionenv]{Corollary}
\newtheorem{propositionenv}[definitionenv]{Proposition}
\newtheorem{conjectureenv}[definitionenv]{Conjecture}
\newtheorem{remarkenv}[definitionenv]{Remark}

\newenvironment{remark}{\begin{remarkenv}\rm}{\end{remarkenv}}
\newcommand{\br}{\begin{remark}}
	\newcommand{\er}{\end{remark}}

\newtheorem{exampleenv}{Example}
\newtheorem{app-lemmaenv}[section]{Lemma}

\newenvironment{definition}{\begin{definitionenv}\rm}{\end{definitionenv}}
\newenvironment{lemma}{\begin{lemmaenv}\rm}{\end{lemmaenv}}
\newenvironment{theorem}{\begin{theoremenv}\rm}{\end{theoremenv}}
\newenvironment{corollary}{\begin{corollaryenv}\rm}{\end{corollaryenv}}
\newenvironment{example}{\begin{exampleenv}\rm}{\end{exampleenv}}
\newenvironment{proposition}{\begin{propositionenv}\rm}{\end{propositionenv}}
\newenvironment{conjecture}{\begin{conjectureenv}\rm}{\end{conjectureenv}}
\newenvironment{app-lemma}{\begin{app-lemmaenv}\rm}{\end{app-lemmaenv}}

\newcommand{\bd}{\begin{definition}}
	\newcommand{\ed}{\end{definition}}
\newcommand{\bl}{\begin{lemma}}
	\newcommand{\el}{\end{lemma}}
\newcommand{\elp}{\hspace*{\fill} $\Box$
\end{lemma}}
\newcommand{\bt}{\begin{theorem}}
\newcommand{\et}{\end{theorem}}
\newcommand{\etp}{\hspace*{\fill} $\Box$
\end{theorem}}
\newcommand{\bc}{\begin{corollary}}
\newcommand{\ec}{\end{corollary}}
\newcommand{\ecp}{\hspace*{\fill} $\Box$
\end{corollary}}
\newcommand{\bcj}{\begin{conjecture}}
\newcommand{\ecj}{\end{conjecture}}

\newcommand{\be}{\begin{example}}
\newcommand{\ee}{\end{example}}
\newcommand{\eep}{\hspace*{\fill} $\Box$
\end{example}}
\newcommand{\bp}{\begin{proposition}}
\newcommand{\ep}{\end{proposition}}
\newcommand{\epp}{
\end{proposition}}
\newcommand\nd{\noindent}

\newcommand{\sD}{\mathscr{D}}

\newcommand{\cF}{\mathcal{F}}

\newcommand{\cH}{\mathcal{H}}

\newcommand{\cJ}{\mathcal{J}}

\newcommand{\reals}{{\mathbb {R}}}
\newcommand{\ff}{{\mathbb F}}

\usepackage{xargs}
\usepackage[colorinlistoftodos,prependcaption]{todonotes}

\allowdisplaybreaks

\newcommandx{\rednote}[2][1=]{\todo[linecolor=red,backgroundcolor=red!25,bordercolor=red,#1]{#2}}
\newcommandx{\bluenote}[2][1=]{\todo[linecolor=blue,backgroundcolor=blue!25,bordercolor=blue,#1]{#2}}
\newcommandx{\yellownote}[2][1=]{\todo[linecolor=yellow,backgroundcolor=yellow!25,bordercolor=yellow,#1]{#2}}
\newcommandx{\greennote}[2][1=]{\todo[inline,linecolor=olive,backgroundcolor=green!25,bordercolor=olive,#1]{#2}}
\newcommand{\rmark}[1]{{\color{red} #1}}

\begin{document}
\title[]{On the size of maximal binary codes with 2, 3, and 4 distances}

\author[]{Alexander Barg$^1$}\thanks{$^1$Department of ECE and ISR, University of Maryland, College Park, MD 20742, USA. Email: abarg@umd.edu.}
\author[]{Alexey Glazyrin$^2$}\thanks{$^2$School of Mathematical and Statistical Sciences, University of Texas Rio Grande Valley, Brownsville, TX 78520, USA. Email: alexey.glazyrin@utrgv.edu}
\author[]{Wei-Jiun Kao$^3$}\thanks{$^3$Institute of Mathematics, Academia Sinica, Taipei 10617, Taiwan. Email: nonamefour0210@gmail.com}
\author[]{Ching-Yi Lai$^4$}\thanks{$^4$Institute of Communications Engineering, National Yang Ming Chiao Tung University (NYCU), Hsinchu 30010, Taiwan, and Physics Division, National Center for Theoretical Sciences, Taipei 10617, Taiwan.
Email: cylai@nycu.edu.tw.} 
\author[]{Pin-Chieh Tseng$^5$}\thanks{$^5$Institute of Communications Engineering and the Department of Applied Mathematics, National Yang Ming Chiao Tung University (NYCU), Hsinchu 30010, Taiwan.
Email: pichtseng@gmail.com.}
\author[]{Wei-Hsuan Yu$^6$}
\thanks{$^6$Department of Mathematics, National Central University, Taoyuan 32001, Taiwan. Email: u690604@gmail.com.}

\begin{abstract}
We address the maximum size of binary codes and binary constant weight codes with few distances. Previous works established a number of bounds for these quantities as well as the exact values for a range of small code lengths. As our main results, we determine the exact size of maximal binary codes with two distances for all lengths $n\ge 6$ as well as the exact size of maximal binary constant weight codes with 2,3, and 4 distances for several values of the weight and for all but small lengths. 
\end{abstract}

\maketitle

\section{Introduction}
Let $M$ be a metric space with metric $d$ that defines the distance between any pair of points in $M$. 
For  a finite set $X \subset M$, the set of distinct distances in $X$, defined as $ \sD(X)=\{ d(x, y) : x\neq  y \in X\},$
is called a \textit{distance set} of $X$.
We say that $X$ is an \textit{$s$-distance set} in $M$ if  $|\sD(X)|=s$. 
For instance, the eight vertices of a cube in $\mathbb{R}^3$ form a three-distance set.

Denote by $A(M,s)$ the maximum cardinality of an $s$-distance set in $M$. 
If $\sD=\{a_1,\dots,a_s\}$ is a set of allowed distances, then we may also use a more detailed notation $A(M,\sD).$ When the distances are specified explicitly, e.g., $\sD=\{a,b\},$ we also write $A(M, \{a,b\})$ for the maximum size of s-distance sets in the space $M$ with distances $a$ and $b$.

The maximum size of $s$-distance sets has been studied for a range of metric spaces. The first to study it were Einhorn and Schoenberg~\cite{ES66I},~\cite{ES66II} who addressed the case of the
$n$-dimensional Euclidean space $M=\reals^n$. This problem is still far from being resolved. As an example, it was proved
only recently that the only maximum three-distance set in $\mathbb{R}^3$ is formed by the 12 vertices of a regular icosahedron, and thus $A(\mathbb{R}^3, 3)=12$; similarly $A(\mathbb{R}^3,5)=20,$ attained by the vertex set of a dodecahedron (see \cite{NozakiShinohara2021}, which also contains other examples). 

The problem of bounding the maximum cardinality of $s$-distance sets has been actively studied for the case of $M=S^{n-1}(\reals),$ the
unit sphere in the $n$-dimensional real space. Delsarte, Goethals, and Seidel \cite{DGS77} proved a general upper estimate
on the size of spherical $s$-distance sets called the {\em harmonic bound}. Their results were improved in a series of recent papers
\cite{M09,Noz11,BM11,MN11,barg2013two,GY18,JTY+20,liu2020semidefinite} and other works that focused in particular on the case of $s=2,3,4$. The case of two-distance 
spherical codes is the most studied one. Apart from being of interest in their own right, spherical 2-distance
sets exhibit deep connections with a number of well-known objects in algebraic combinatorics, notably strongly regular graphs, tight frames, and spherical designs \cite{DGS77,BB05,BGOY15}. Recently the problem of bounding the maximum size of such codes was resolved for most values of the dimension $n$ \cite{GY18}. Bounds obtained in this paper form a starting point of one of our theorems, and we provide more details on it in the main text below. Further narrowing down the problem, many papers
have addressed the maximum size of spherical 2-distance sets in which every pair of vectors have inner product $\alpha$ or $-\alpha$ for some value $\alpha\in(0,1),$ also known as sets of equiangular lines. This problem also enjoys long history in the literature with a 
number of links in algebraic combinatorics, and we refer the reader to \cite{barg2013new,GY18,BDKS18,JTY+21} for recent advances and a general overview.

The maximum cardinality problem for $s$-distance sets has also been studied for metric spaces commonly considered in coding theory, 
starting with the Hamming space $\cH_q^n$ of $n$-dimensional vectors over a finite field $\ff_q.$ Linear two-distance codes in particular
form a classic object in coding theory because of their links to finite geometries, see for instance an early survey \cite{CK86} as well as recent works \cite{SHS+21}; \cite[Ch.7]{BvM22}. Linear codes with few weights are also often
constructed in the framework of cyclic codes and Boolean functions, in particular, bent functions \cite{M18}. 
This line of research is a subject of vast literature, starting with classical families such as duals of 2-error-correcting BCH codes, Kerdock codes and their relatives \cite{MS77} and many other families. Some recent additions to the literature are \cite{WZHZ15,D16,LLHQ21} (the reader is cautioned that this sample is far from complete in many respects).

At the same time, most of the code families studied in the context of coding theory limit not just on the number of distances but also
aim at constructing codes with large minimum distance. As a result, the codes thus obtained are of cardinality that is much smaller than 
the maximum possible if the latter constraint is removed. It is this problem that we seek to resolve in this work: what is the maximum size
of a binary code with two distances (linear or not) when the minimum separation between distinct codewords is unrestricted?
In some cases, notably for binary codes with two distances we give a complete answer, where previously the exact values were 
known only for a range of small code lengths $n$.

In addition to $\cH_{2}^{n}$ we study codes in the binary Johnson space $\cJ^{n, w},$ formed of all binary vectors of length $n$ and Hamming weight $w$ such that $2w \leq n$. For two vectors $x,y\in \cJ^{n,w}$ we define the {\em Johnson distance} $d_J(x,y)=(1/2)d_H(x,y),$ where $d_H(x,y)=|\{i: x_i
\ne y_i\}|$ is the Hamming distance. The problem of determining the maximum size of codes in $\cJ^{n,w}$ with few distances recently was the subject of several papers in coding theory literature \cite{BM11,MN11}. This problem can also be phrased in different terms that highlight connections in extremal combinatorics. A code $C\subset \cJ^{n,w}$ can be thought of as a family of $w$-subsets of an $n$-set, and $|x\cap y|=w-d_J(x,y),$ so disallowing some values of the distance is tantamount to forbidding some intersections of the subsets in the family. For instance, the celebrated Erd{\H o}s-Ko-Rado theorem \cite{EKR61}; \cite{FT18,GM16} is a statement about the maximum size of a code in $\cJ^{n,w}$ with $\sD=\{1,2,\dots,w-t\},$ where $t\ge1$ is some fixed number. Generally, the problem of estimating the maximum size of families of finite sets with a prescribed set of intersections has been the subject of a long line of works in combinatorics; we refer to recent overviews in \cite{FT18,E22}.

Following the convention of coding theory, we call an $s$-distance set in the Hamming and Johnson spaces an \textit{$s$-code}. The maximum cardinality of an $s$-code can be bounded above using Delsarte's linear programming (LP) method \cite{D73}; however this approach 
typically does not yield closed-form results. The same applies to Schrijver's semidefinite programming (SDP) method \cite{S05,GMS12} which represents a
far-reaching extension of the LP bound. Delsarte \cite{D73,Del73b} also proved a general upper bound on the size of $s$-codes, called the {\em harmonic bound}; see also the paper \cite{NS10} for a recent improvement. Relying on the harmonic bound, the authors of \cite{BM11} determined the maximum cardinality $A(\cH_{2}^{n}, 2)$ and $A(\cJ^{n, w}, 2)$ in a number of cases, and paper \cite{MN11} extended their results to $s = 3$ and $4$. 

Despite the described advances, the exact values of $A(M,s)$ for $M=\cH_2^n$ or $\cJ^{n,w}$ were known only for a limited range
of values of $n$. For instance, it was shown in \cite{BM11} that $A(\cH_2^n,2)=1+\binom n2$ for $6\le n\le 78$ with a number of
exceptions in that range left unresolved, and \cite{MN11} showed that $A(\cH_2^n,3)= n + \binom{n}{3}$ for $8\le n\le 37$ and $n=44$,
with the exception of $n=23,34,35$ (while $n=23$ constitutes a true exception due to the fact that the dual Golay code ${\mathcal G}_{23}^\bot$ yields $A(\cH_{2}^{23}, 3)\ge 2048$, the other two values are not as we will observe below). As one of our main results we prove
\begin{theorem}\label{thm:Hamming2}
$A(\cH_2^n,2)=1+\binom n2$ for all $n\ge 6.$
\end{theorem}
The idea of the proof is to embed a binary code into a sphere $S^{n-1}$ and relate its size to spherical two-distance sets and
spherical equiangular sets, relying on bounds for them proved recently in \cite{GY18}.

\begin{table}[t]\label{table:gb}
\caption{{\small Values of $A(J^{n,w},s)$}}
{\small \begin{tabular}{|c|*{10}{|c}|}
\hline
$w$&     4  &5 &6 &5 &6 &7 &6 &7 &any\\[.03in]
$s$ &    2  &2 &2 &3 &3 &3 &4 &4 &$w-1$\\[.03in]
$A(\cJ^{n,w},s)= $ & $\binom{n-2}2$& $\binom{n-3}2$ &$\binom{n-4}2$ &$\binom{n-2}3$ &$\binom{n-3}3$ &$\binom{n-4}3$ &$\binom{n-2}3$ &$\binom{n-3}3$ &$\binom{n-1}{w-1}$ \\[.03in]
for $n\ge$ &9 &12 &35 &12 &16 &{20} &15 &20 &$2w$\\[.04in]
Thm/Cor &\ref{thm:Johnson2}(a) &\ref{thm:Johnson2}(b) &\ref{thm:Johnson2}(c)&\ref{thm:Johnson3}(b)&\ref{thm:Johnson3}(c)&\ref{thm:Johnson3}(d)&\ref{thm:Johnson4}(b)&\ref{thm:Johnson4}(c)&\ref{cor:ww-1}\\
\hline
\end{tabular}}
\end{table}

{\begin{remark} The maximum size of 2-distance codes in $\cH_q^n$ was recently studied in \cite{BDZZ21} which put forward a conjecture that 
   $$
A(\cH^{n}_{q}, \{2, 4\})=1+\binom n2 \text{ for } n \geq 6 \text{ and } q = 2, 3, 4.
   $$
We note that for $q=2$ this equality follows already from Theorem~\ref{thm:ham_up}
(an independent different proof was recently given in \cite{LR21}).
Our Theorem \ref{thm:Hamming2} extends this result to all values of the distances.
\end{remark}}

We further establish a number of exact general results for $A(\cJ^{n,w},s)$ for $3\le w\le 7$ and $2\le s\le 4$, detailed in Section \ref{sec:small-w}. Here the main vehicle is application of the LP method. We remarked above that general resuls based on it are difficult to come by; an observation
that makes our results possible is that a small subset of Delsarte's inequalities suffice to constrain the
size of the code, and this small subset can be analyzed in a general form rather than numerically. Outside the above range of values
of $s$ and $w$ we obtain a few new exact results by computer relying on LP as well as on Schrijver's SDP bound.

For convenience we list the new bounds for $A(\cJ^{n,w},s)$ in Table~\ref{table:gb} with references to the statements proved in the main text. {Earlier results \cite{MN11} additionally show that $A(\cJ^{n,6},2)=\binom{n-4}2$ for $15\le n\le 24.$  
We discuss these as well as other previously known and new numerical results in Sec.~\ref{sec:numerical} below.}

Based on our new results and known bounds (including the Erd\H{o}s-Ko-Rado theorem), we formulate a general conjecture for $s$-distance sets in the Johnson space.

\begin{conjecture}\label{conj:Johnson_general}
For all $w>s>0$ and sufficiently large $n\geq n_0(w,s)$, $$A(\cJ^{n,w},s)=\binom{n-w+s}s.$$
\end{conjecture}
The Erd\H{o}s-Ko-Rado theorem corresponds to the case of $s$ {\em consecutive} distances
$\sD=\{1,2,\dots,s\},$ see Sec.~\ref{sec:forbidden}.

\section{The harmonic bound, its improvements and implications}
\subsection{The Hamming space}

Recall that the set of polynomials
\begin{equation}\label{eq:Krawtchouk}
    \phi_{k}(x) = \sum_{j = 0}^{k} (-1)^{j}\binom{x}{j}\binom{n-x}{k-j}, k=0,1,\dots,n
\end{equation}
called the {\em binary Krawtchouk polynomials}, forms an orthogonal basis in the space of real functions on $\{0,1,\dots,n\}$ with weight $\binom ni,i=0,\dots,n$  \cite{D73},\,\cite[Ch.5]{MS77}. Any such function has a unique expansion as a sum of Krawtchouk polynomials, namely $f(x)=\sum_{i\ge 0} f_i \phi_i(x),$ where $f_i=2^{-n}\sum_{j\ge0} f(j)\phi_i(j)\binom nj.$

Delsarte \cite{D73,Del73b} proved an upper bound on $A(\cH_2^n,s)$, called the {\em harmonic bound} 
(it is derived relying on the dimension of the space of spherical harmonics). 
In the next theorem we quote this bound accounting for an improvement found in \cite{NS10}. 
\begin{theorem}\cite{NS10} \label{thm:HB}
Let $C \subset \cH_{2}^{n}$ be an $s$-code with distances $\sD=\{d_{1}, \dots, d_{s}\}$. Consider the polynomial $f(t) = \prod_{i} \frac{d_{i}-t}{d_{i}}= \sum_{k = 0}^{s}f_{k}\phi_k(t)$. Then
\[
A(\cH_2^n,\sD) \leq \sum_{k: f_{k} > 0} \binom{n}{k}.
\]
\end{theorem}
In \cite{NS10} this theorem is proved for spherical codes, but it can be extended to other polynomial metric spaces including
$\cH_2^n$ and $\cJ^{n,w}$ using standard tools.

A concrete form of Theorem \ref{thm:HB} under the condition that the sum of 
the code's distances does not exceed $\frac{1}{2}sn$ is as follows.
\begin{theorem}\cite[Theorem $11$]{BM11}
\label{thm:ham_up}
Suppose that $\sD=\{d_1,\dots,d_s\}$ is such that $\sum_{i = 1}^{s} d_{i} \leq \frac{1}{2}sn.$ Then
\begin{equation*}
    A(\cH_2^n,\sD) \leq \sum_{i = 0}^{s-2} \binom{n}{i} + \binom{n}{s}.
\end{equation*}
\end{theorem}
This inequality is proved by showing that the assumption $ \sum_{i = 1}^{s} d_{i} \leq \frac{1}{2}sn$ 
forces the coefficient $f_{s-1}$ of the Krawtchouk expansion to be nonpositive, and thus the term $\binom n{s-1}$ 
is missing from the sum. 

A lower bound on $A(\cH_2^n,s)$ follows by an easy construction of $s$-codes in $\cH_{2}^{n}$ which sometimes 
enables one to prove tight results.

\begin{proposition}\cite[Eq.$(3.1)$]{MN11} 
\label{prop_ham_low}
For $2s \leq n$, we have
\begin{equation*}
    A(\cH_{2}^{n}, s) \geq \sum_{i = 0}^{\lfloor \frac{s}{2} \rfloor} \binom{n}{s-2i}.
\end{equation*}
\end{proposition}

\begin{proof}
For  $2s \leq n$,  the set $C$ of all binary vectors of Hamming weight  $k\leq s$ such that 
$k=s~(\text{mod}\, 2)$ is an $s$-code with distance set $\sD(C)=\{2, 4, \dots, 2s\}$.
\end{proof}

\subsection{The Johnson space} The results in the previous section have their analogs for the 
Johnson space $\cJ^{n, w}.$ 
In this case the relevant family of orthogonal polynomials are the {\em Hahn polynomials,} defined as
   \begin{equation}\label{eq:Hahn}
    \psi_{k}(x) = \sum_{j = 0}^{k} (-1)^{j}\frac{\binom{k}{j} \binom{n+1-k}{j}}{\binom{w}{j} \binom{n-w}{j}} \binom{x}{j},\quad k=0,1,\dots,w.
   \end{equation}
The Hahn polynomials form an orthogonal basis in the function space $f:\{0,1,\dots,w\}\to {\mathbb R}$ with weight
$\binom wi\binom{n-w}i,i=0,1,\dots,w.$

The (improved) harmonic bound for the Johnson space has the following form. 
\begin{theorem}\cite{NS10}\label{thm:JNS}
Let $C \subset \cJ^{n, w}$ be a constant weight $s$-code with distances $\{d_{1}, \dots, d_{s}\}$. Consider the polynomial $f(t) = \prod_{i} \frac{d_{i}-t}{d_{i}}$. Suppose that the Hahn expansion of $f(t)$ is   $f(t) = \sum_{k = 0}^{s}f_{k}\psi_k(t)$. Then
    $$
|C| \leq \sum_{k: f_{k} > 0} \binom n{k-1}\frac{n-2k+1}k.
    $$
\end{theorem}
As remarked, this result appears in \cite{NS10} for the case of spherical codes. Starting with this theorem, one can prove 
the following upper bound on $s$-codes in the Johnson space.

\begin{theorem}\cite[Theorem $8$]{BM11}
\label{thm:Jon_BM_upp}
Let $C \subset \cJ^{n, w}$ be an $s$-code with distances $\sD=\{d_{1}, \dots, d_{s}\}$. Suppose that
\begin{equation*}
    \sum_{i = 1}^{s} (w-d_{i}) \geq \frac{s(w^{2}-(s-1)(2w-\frac{n}{2}))}{n-2(s-1)}.
\end{equation*}
Then
\begin{equation*}
    |C| \leq \binom{n}{s} - \binom{n}{s-1} \frac{n-2s+3}{n-s+2}.
\end{equation*}
\end{theorem}

Similarly to Proposition~\ref{prop_ham_low}, there is a general construction of $s$-codes in the Johnson space that implies a lower bound on $A(\cJ^{n, w}, s)$.
\begin{proposition}\cite{EKR61} 
\label{prop:MN_J2}
For $s \leq n-w$ we have
\begin{equation*}
    A(\cJ^{n, w}, s) \geq \binom{n-w+s}{s}.
\end{equation*}
\end{proposition}
\begin{proof}
For $s \leq n-w$,  consider the set of binary vectors in $\cJ^{n,w}$ with ones in the first $w-s$ coordinates. 
Clearly, it forms an $s$-code of size $\binom{n-w+s}{s}$ with distances $\{1,\dots,s\}$.
\end{proof}

\section{The maximum size of \texorpdfstring{$2$}--codes in the Hamming space: A proof of Theorem \ref{thm:Hamming2}}
We will prove Theorem \ref{thm:Hamming2} by reduction to spherical codes.
Let $C\subset \cH_2^n$ be a binary code of length $n.$ Our plan is to map $C$ on the sphere $S^{n-1}$ and view the result as a spherical $2$-distance code. 
For that let $f:\cH_2^n\to S^{n-1}$ be a map given by $x = (x_{1}, \dots, x_{n}) \mapsto f(x) = (v_{1}, \dots, v_{n})$, where $v_{i} = \frac{1}{\sqrt{n}}(-1)^{x_{i}}$ for all $i$. Note that if $d_H(x,y)=d$ then for the Euclidean inner product we obtain
$\langle f(x), f(y) \rangle=1-\frac{2d}{n}$. Thus, if $C$ is a binary 2-code with distances $a$ and $b$, the set $f(C)$ forms 
a spherical $2$-code with inner products $\alpha = 1-\frac{2a}{n}$, $\beta = 1-\frac{2b}{n}$.

In the following subsection we recall some known bounds for equiangular lines. The proof itself is formed of two parts presented in Sections~\ref{sec:EL}, \ref{sec:EC}. The first part reformulates the problem for spherical codes,
connects it to families of equiangular lines, and then narrows down the possible parameters by using the known results about them. This leaves behind a set of exceptional values of the distances, which we handle in the second part. 

\subsection{Bounds from spherical equiangular sets} \label{sec:EL}

We will need a few known results concerning equiangular line sets in $\reals^n.$ Any such set is a finite collection of points
on $S^{n-1}$ with pairwise inner products $\{\alpha, -\alpha\}$, $\alpha\in[0,1)$. Let $M(n)$ be the size of the largest possible equiangular line set in $n$ dimensions and let $M_{\alpha} (n)$ be the size of the largest equiangular set with inner products $\{\alpha, -\alpha\}$.

\vspace*{.1in}
\begin{theorem}\label{thm:M_alpha}
\hspace*{-.2in}\begin{align*}
 \hspace*{-.2in}  (a)\; & M_{\alpha}(n)  \leq 2n  \text{ unless $1/\alpha$ is an odd integer}  \text{ (Neumann, see \cite{LS73})}; \\[.1in]
   (b) \;& M_{\frac 1 3}(n) = 28 \text{ for } 7 \leq n \leq 15, \text{ and } M_{\frac 1 3}(n) = 2n-2 \text{ for } n \geq 15.\text{ (\cite{LS73})};\\[.1in]
  (c)\; &M_{\alpha}(n)  \leq \frac {n(1-\alpha^2)}{1-n\alpha^2}  \text{ for all $\alpha$ and $n \in \mathbb{N}$ such that $n\alpha^2 < 1$}  \text{ (Relative bound, \cite{LS73})};\\[.1in]
   (d) \; & M_{\frac 1 a}(n) \leq  \frac {(a^2-2)(a^2-1)} 2 \text{ for all $n$ and $a$ such that $n \leq 3a^2-16$ and $a \geq 3$ }
   \text{(\cite{Yu17})};\\[.1in]
  (e) \; &M_{\frac 1 a}(n)  \leq n\left(\frac 2 3 a^2 + \frac {4} {7}\right)+2  \text{ for all $a\geq 3$ and for all $n \in \mathbb{N}$ (\cite[Theorem 3]{GY18}).}
\end{align*}
\end{theorem}

Parts (a)-(c) of this theorem form classical results while parts (d) and (e) represent recent results on bounds for equiangular lines.

\vspace*{.1in}
\begin{lemma}
\label{lem_equ}
$M_{\frac 1 a}(n) \leq \binom{n-1}{2}$ for $n\geq 7$ unless $a$ is an odd integer and $n \in \{a^{2}-1, a^{2}-2\}$.
\end{lemma}

\begin{proof}
The case-by-case proof below essentially follows the proof of \cite[Theorem 2]{GY18}.
\begin{enumerate}
\item
If $a$ is not an odd integer, by Theorem \ref{thm:M_alpha}(a), $M_{\frac 1 a}(n) \leq 2n \leq \binom{n-1}{2}$ for all $n\geq 7$. Therefore,
below we
will assume that $a$ is an odd integer.
\item
Let $a^{2} \leq n \leq 3a^{2}-16.$ By Theorem \ref{thm:M_alpha}(d), 
\begin{equation*}
    M_{\frac 1 a}(n) \leq \frac{(a^{2}-1)(a^{2}-2)}{2} \leq \binom{n-1}{2}.
\end{equation*}
\item
For $n \leq a^{2}-3$, by Theorem \ref{thm:M_alpha}(c),
\begin{equation*}
    M_{\frac 1 a}(n) \leq \frac{n(a^{2}-1)}{a^{2}-n} = n + \frac{n^{2}-n}{a^{2}-n} \leq n+\frac{n^{2}-n}{3}\leq \binom{n-1}{2}
\end{equation*}
for all $n\geq 13$. If $7 \leq n \leq 12$, then $a \geq 5$ (recall that $a$ is odd) and, again using Theorem \ref{thm:M_alpha}(c), we obtain
\begin{equation*}
    M_{\frac 1 a}(n) \leq n + \frac{n^{2}-n}{a^{2}-n} \leq n + \frac{n^{2}-n}{13} \leq \binom{n-1}{2}.
\end{equation*}
\item
For the last remaining case, consider $n \geq 3a^{2}-15$. By Theorem \ref{thm:M_alpha}(e),
\begin{equation*}
    M_{\frac 1 a}(n) \leq n\left(\frac{2}{3}a^{2} + \frac{4}{7}\right) + 2 \leq n \left(\frac{2}{3} \cdot \frac{n+15}{3}+\frac{4}{7} \right)+2 \leq \binom{n-1}{2}
\end{equation*}
for all $n\geq 20$. For $12 \leq n \leq 19$ we have $a = 3$, and by Theorem \ref{thm:M_alpha}(b)
\begin{equation*}
    M_{\frac 1 a}(n) = 28 \leq \binom{n-1}{2}.
\end{equation*}

\end{enumerate}
\end{proof}

Recall that $A(\cH_{2}^{n}, \{a, b\})$ is the maximum size of $2$-codes in the Hamming space with distance set $\{a, b\}$ and $a > b$, and let $g(n, \alpha, \beta)$ be the maximum size of a spherical $2$-code with inner products $\alpha, \beta \in [-1, 1)$.

\begin{lemma}
\label{lem_sph}
Suppose that $-1 \leq \alpha < \beta < 1$ and $\alpha + \beta < 0$. Let $\gamma = \frac{2-(\alpha+\beta)}{\beta - \alpha}$. Then for $n\geq 6$,
\begin{equation*}
    g(n, \alpha, \beta) \leq \binom{n}{2}
\end{equation*}
unless $\gamma$ is an odd integer and $n \in \{\gamma^{2}-2, \gamma^{2}-3\}$.
\end{lemma}

\begin{proof}
Consider a $2$-code $C = \{x_{1}, \dots, x_{N}\}\subset S^{n-1}$ with inner products $\alpha, \beta$. Let $t = \sqrt{\frac{2}{2-(\alpha+\beta)}}$ and let $y \in S^{n}$ be the unit vector orthogonal to all $x_{i} \in C$. Then the set
\begin{equation*}
    \{tx_{i} + \sqrt{1-t^{2}}y : 1 \leq i \leq N\}
\end{equation*}
is an equiangular set in $S^{n}$ with inner product $\frac{1}{\gamma} = \frac{\beta-\alpha}{2-(\alpha+\beta)}$. Thus, 
\begin{equation*}
    g(n, \alpha, \beta) \leq M_{\frac{1}{\gamma}}(n+1) \leq \binom{n}{2}
\end{equation*}
unless $\gamma$ is one of the exceptional values in the statement of the Lemma \ref{lem_equ}.
\end{proof}

\begin{theorem}
\label{thm:ham_partial_pro}
For $n\geq 6$,
\begin{equation*}
    A(\cH_{2}^{n}, \{a, b\}) \leq 1 + \binom{n}{2}
\end{equation*}
unless there is a positive integer $m$ such that $a = (m+1)(2m+1)$, $b = m(2m+1)$ and $n \in \{(2m+1)^{2}-2, (2m+1)^{2}-3\}$.
\end{theorem}

\begin{proof}
If $a + b \leq n$, by Theorem \ref{thm:ham_up}, $A(\cH_{2}^{n}, a, b) \leq 1 + \binom{n}{2}$. Now we assume $a + b > n$ and consider the corresponding spherical $2$-code with scalar products $\alpha = 1-\frac{2a}{n}$, $\beta = 1-\frac{2b}{n}$. Then $\alpha + \beta < 0$. By Lemma \ref{lem_sph},
\begin{equation*}
    A(\cH_{2}^{n}, a, b) \leq g(n, \alpha, \beta) \leq \binom{n}{2}
\end{equation*}
unless $\gamma = \frac{2-(\alpha+\beta)}{\beta - \alpha}$ is an odd integer and $n \in \{\gamma^{2}-2, \gamma^{2}-3\}$. Write $\gamma = 2m + 1$ for some integer $m$ and observe that
\begin{equation*}
    \gamma = \frac{a+b}{a-b} = 2m+1
\end{equation*}
implies $\frac{a}{b} = \frac{m+1}{m}$. Denoting $k=\text{gcd}(a,b)$, we get $a=(m+1)k$ and $b=mk$.

By a result of Rankin \cite{Ran55}, the size of a spherical code whose scalar products are all negative cannot be larger than $n+1$. Since $n+1\leq \binom{n}{2}$, it is sufficient to consider the case when $\beta\geq 0$, that is, when $b\leq n/2$. Note that $\alpha+\beta<0$ so $\alpha$ must be negative and $a>n/2$. Using $n \in \{(2m+1)^{2}-2, (2m+1)^{2}-3\}$, we see that $b\leq 2m^2+2m-1$ and $a\geq 2m^2+2m$. Since $b=mk$, $k< 2m+2$. Since $a=(m+1)k$, $k\geq 2m$. The only possible choices are $k=2m$ and $k=2m+1$.

Note that $\alpha+\beta<0$ is equivalent to $a+b>n$. If $k=2m$, $a+b=2m(m+1)+2m^2=4m^2+2m\leq (2m+1)^2-3\leq n$ so this condition is not satisfied. Therefore, the only remaining choice for $k$ is $2m+1$ and $a=(m+1)(2m+1)$, $b=m(2m+1)$.
\end{proof}

\subsection{Exceptional cases} \label{sec:EC}

Here we show that the exceptional cases in Theorem~\ref{thm:ham_partial_pro} cannot give rise to spherical $2$-codes of large size. The main instrument in the proof is a bound on spherical $2$-codes from \cite{GY18}.

\begin{theorem}\cite[Corollary 4]{GY18}\label{thm:2-dist}
Let $X\subset S^{n-1}$ be a spherical code with inner products $\alpha,\beta$. Then
$$|X|\leq \dfrac {n+2}{1-\frac {n-1}{n(1-\alpha)(1-\beta)}},$$
if the right-hand side is positive.
\end{theorem}

\begin{lemma}\label{lem:no_exceptions}
For all $n\geq 6$,
\begin{equation*}
    A(\cH_{2}^{n}, \{a, b\}) \leq \binom{n}{2},
\end{equation*}
given that $a = (m+1)(2m+1)$, $b = m(2m+1)$ and $n \in \{(2m+1)^{2}-2, (2m+1)^{2}-3\}$ for some positive integer $m$.
\end{lemma}

\begin{proof}
Using the same mapping to the sphere as above, we consider spherical $2$-codes with scalar products $\alpha=1-\frac {2a} {n}$ and $\beta=1-\frac {2b} n$. Note that
$$n(1-\alpha)(1-\beta) = \frac {4ab} {n} = \frac {(2m+1)^4-(2m+1)^2} {n}.$$
Since $(2m+1)^2$ is either $n+2$ or $n+3$,
$$n(1-\alpha)(1-\beta)\geq \frac {(n+2)^2-(n+2)} {n} = \frac{(n+2)(n+1)} {n}.$$

Therefore, by Theorem \ref{thm:2-dist}, the size of the set is no larger than
$$\frac {n+2} {1 - \frac {(n-1)n} {(n+2)(n+1)}} = \frac {(n+1)(n+2)^2} {4n+2} < \binom{n}{2} + 1$$
for all $n\geq 7$. One can also check that the upper bound in Theorem \ref{thm:2-dist} is less than $\binom{n}{2}$ when $n = 6$, $m = 1$.
\end{proof}

Now Theorem~\ref{thm:Hamming2} follows by combining Proposition \ref{prop_ham_low}, Theorem \ref{thm:ham_partial_pro} and Lemma \ref{lem:no_exceptions}.

\begin{remark}
Another way to rule out the existence of large-size codes for the exceptional sets of parameters is to rely on the LP method. This route relies on combining the Delsarte inequalities \eqref{eq:DH} for degrees $k=2,n-2,n-1,n$, but it results
in a longer argument so we do not present it here. Similar (yet different) arguments are used to bound $A(J^{n,w},s)$
in Section~\ref{sec:small-w}.
\end{remark}

\begin{remark}
As follows from \cite{JTY+20}, for large $n$ the value $A(\cH_2^n,\{d_1,d_2\})$ for some distance pairs $d_1,d_2$ is at most linear in $n.$ These results rely on certain technical conditions related to spectral radius of some graphs which we do not cite here. To give an example, for large $n$ the value 
$A(\cH^n_2, \{\frac {3n} {10}, \frac {3n} 5\}) \leq 3n +O(1)$, see \cite[Thm.~1.12]{JTY+20}.

\end{remark}

\section{Bounds for \texorpdfstring{$s$}--codes in the Johnson space}

\subsection{Implications of the harmonic bound}
In this section we derive some general bounds for $A(\cJ^{n,w},s)$ based on a refining of Theorem \ref{thm:Jon_BM_upp}, which 
we state an proof next.
\begin{theorem}
\label{thm:Joh_general}
Let $C \subset \cJ^{n, w}$ be an $s$-code with distance set $\sD=\{d_{1}, \dots, d_{s}\}$, where $w \geq s+1$ and $s \geq 2$. Suppose that 
  \begin{equation}\label{eq:25}
  n \geq sw^{2}-2s(s-1)w+s(s-1)^{2}+2(s-1).
  \end{equation}
Then $|C| \leq \binom{n}{s}-\binom{n}{s-1}+\binom{n}{s-2}$.
\end{theorem}

\begin{proof}
From \eqref{eq:25} we conclude that
\begin{equation*}
    \frac{sw(n-2(s-1))-s(w^{2}-(s-1)(2w-\frac{n}{2}))}{n-2(s-1)} \geq \sum_{i=0}^{s-1}(w-i) - 1.
\end{equation*}
By Theorem \ref{thm:Jon_BM_upp}, if 
\begin{equation*}
    \sum_{i=1}^{s} d_{i} \leq \frac{sw(n-2(s-1))-s(w^{2}-(s-1)(2w-\frac{n}{2})))}{n-2(s-1)},
\end{equation*}
$|C| \leq \binom{n}{s}-\binom{n}{s-1}\frac{n-2s+3}{n-s+2} = \binom{n}{s}-\binom{n}{s-1}+\binom{n}{s-2}$. 
This proves our claim except for the case of $\sD = \{w, w-1, \dots, w-s+1\}$. 
However, for this choice of the distances the Johnson bound \cite{J62}, \cite[p.~528]{MS77} implies that
\begin{equation*}
    |C|\le\prod_{i=0}^{s-1} \frac{n-i}{w-i} \leq \binom{n}{s}-\binom{n}{s-1}+\binom{n}{s-2}.
\end{equation*}
\end{proof}

This theorem implies the following bounds.

\begin{corollary}\label{cor:Johnson_harmonic}
    \begin{align}
A(\cJ^{n, w}, 2) &\leq \binom{n}{2}-n+1 = \binom{n-1}{2}, \quad w\geq 3 \text{ and } n\geq 2w^2-4w+4,\\
    A(\cJ^{n,w},3) &\leq \binom{n}{3}-\binom{n}{2}+n = \binom{n-1}{3}+1, \quad w\ge 4\text{ and }n \geq 3w^{2}-12w+16\\
    A(\cJ^{n,w},4) &\leq \binom{n}{4}-\binom{n}{3}+\binom{n}{2}= \binom{n-1}{4}+n+1, 
    \\ &\hspace*{1in}w\ge 5\text{ and }n \geq 4w^{2}-24w+42.
    \nonumber
\end{align}
\end{corollary}

By fixing the value of $w$ we can obtain specific bounds for the maximum size of $s$-codes.
For instance, taking into account numerical results of Propositions \ref{prop:2_dis_equal_J}, \ref{prop:3_dis_equal_J}, and \ref{prop:4_dis_equal_J} below, we obtain the following results:
  \begin{align}
&A(\cJ^{n, 3},2)=\binom{n-1}{2},\quad n \geq 6 \label{eq:h3}\\
&\binom{n-1}{3} \leq A(\cJ^{n, 4}, 3) \leq \binom{n-1}{3}+1,\quad n \geq 11. \label{eq:h4}\\
&\binom{n-1}{4} \leq A(\cJ^{n, 5}, 4) \leq \binom{n-1}{4}+n+1\quad n \geq 15. \label{eq:h5}
\end{align}
These estimates in fact relate to a special situation when the code is allowed to have all but one distances, or using the
language of intersecting families, forbidding a single intersection. In the next subsection we further explore this point of view,
removing the gap in last two of these estimates. Further, in Sec.~\ref{sec:small-w} we show how these improved bounds can be
obtained using the Delsarte inequalities.

\subsection{EKR and forbidden intersections}\label{sec:forbidden}
As already mentioned, constant weight codes with a specified set of distances can be equivalently described as families of
finite sets with restricted intersections. Switching to this language, in this subsection we prove one simple corollary of the Erd{\H o}s-Ko-Rado theorem that implies bounds for $s$-codes with $w=s+1.$ We will use the following notation. 
 A family $\cF$ of $w$-subsets of $[n]$ is called {\em intersecting} if $F_1\cap F_2\ne\emptyset$ and it is called {\em $t$-intersecting} if $|F_1\cap F_2|\ge t$ for any $F_1,F_2\in \cF.$ Let $m(n,w,t)$ be the maximum size of a $t$-intersecting family. 

\begin{theorem}\cite{EKR61} \label{EKR}
For $1\le t< w$ and $n\ge n_0(w,t)$
   \begin{equation}\label{eq:EKR}
     m(n,w,t)=\binom{n-t}{w-t}.
   \end{equation}
If $t=1$ and $w\le n/2,$ then $m(n,w)\le \binom {n-1}{w-1}.$ If $w<n/2$ then equality holds if and only if there is an $i\in [n]$ such that
$\cF=\{F\subset\binom{[n]}w: i\in F\}.$
\end{theorem}
Clearly the family $\cF=\{F\in \binom{[n]}w: \{1,2,\dots,t\}\subset F\}$ has cardinality $\binom{n-t}{w-t}$ and therefore is 
optimal for $n\ge n_0(w,t)$. The value of $n_0(w,t)$ was found in a series of papers culminating with Wilson's theorem \cite{W84} 
stating that for all $t$
  $$
  n_0(w,t)=(w-t+1)(t+1).
  $$
Rephrasing this for s-codes in $\cJ^{n,w}$, the largest size of a code with distances $\sD=\{1,2,\dots,s=w-t\}$ is 
  \begin{equation}\label{eq:2-EKR}
  A(\cJ^{n,w},\sD)=\binom{n-w+s}{s} \text{ for all }s\le w-1,\;n\ge (s+1)(w-s+1).
  \end{equation}

We will now prove a related result for $s=w-1$ (any single missing distance) through a connection with families with forbidden intersections. Denote by $m(n,w,\bar l)$ the maximum size of a family $\cF$ of $w$-subsets of $[n]$ such that $|F\cap F'|\ne l$ for
any $F,F'\in \cF$.
The well-known Frankl-F{\"u}redi theorem \cite{FF85} asserts that for $w\ge 2l+2$ and sufficiently large $n$,
   $$
   m(n,w,\bar l)\le \binom{n-l-1}{w-l-1}.
   $$
We first establish a simple (weaker) upper bound for $m(n,w,\bar l)$  that applies for all $n\ge 2w-l$ \footnote{In all likelihood, it is available in the literature, however we could not find an immediate reference.}.
\begin{theorem}\label{thm:wl}
Let $1\le l\le w-1$ and $n\ge 2w-l,$ then $m(n,w,\bar l)\le \frac{w-l}{n-l}\binom nw.$
\end{theorem}
\begin{proof}
Let $\cF$ be a family of $w$-subsets of $[n]$ no two of which intersect on $l$ elements. Every set $F\in \cF$ contains $\binom wl$
$l$-subsets, so the total number of $l$-subsets in all sets $F$ is $|\cF|\binom wl.$ Thus on average every $l$-subset is contained 
in $E:=\frac{|\cF|\binom wl}{\binom nl}$ sets $F$. Now fix an $l$-subset $S$ that is contained in at least $E$ subsets of $\cF$. By $\cF^*\subseteq \cF$ denote the set of all subsets in $\cF$ containing $S$. For any $F,F'\in \cF^*$, $|F\cap F'|\ne l$ and thus
$(F\backslash S)\cap (F'\backslash S)\ne \emptyset.$ This means that the collection
$\{F\backslash S:F\in \cF^*\}$ forms an intersecting family in $[n]\backslash S$, and so 
 $E\le \binom{n-l-1}{w-l-1}$ by the EKR theorem. The last inequality implies that
  \begin{gather*}
  |\cF|\le\frac{\binom nl}{\binom wl}\binom{n-l-1}{w-l-1}=\frac1{\binom wl}\frac{w-l}{n-l}\binom nl\binom{n-l}{w-l}
  =\frac{w-l}{n-l}\binom nw. 
    \end{gather*}
\end{proof}

\begin{corollary}\label{cor:ww-1}
$A(\cJ^{n,w}, w-1) = \binom{n-1}{w-1}$ for $n \geq 2w$.
\end{corollary}
\begin{proof} We have 
$$
A(\cJ^{n,w}, w-1)\le \max_l \frac{w-l}{n-l}\binom nw\le \binom{n-1}{w-1}.
$$
At the same time, a $(w-1)$-code all of whose vectors contain a fixed point is of size $\binom{n-1}{w-1},$ hence the equality.
\end{proof}

Concluding this section, note that the problem of the largest size of codes with one forbidden intersection in the $q$-ary Hamming space, $q\ge 3$, was recently 
resolved in \cite{KLLM21}.

\section{Using the Delsarte inequalities: Bounds for small \texorpdfstring{$w$}. }\label{sec:small-w}

\subsection{Preliminaries}
A powerful tool to estimate the maximum size of a code is provided by Delsarte's linear programming bound. We state this bound for
$\cH_2^n$ and $\cJ^{n,w}$ in the following theorem, adapting the statements to $s$-codes.

\begin{theorem}\cite{D73} \label{thm:del}
(a) Let $C \subset \cH_{2}^{n}$ be an $s$-code with distances $\sD= \{d_{1}, \dots, d_{s}\}$. Then $|C|$ is bounded above
by the value of the linear program $\text{LP}_{H}(\sD)$ given by
    \begin{equation}\label{eq:LPH}
    1+\sum_{j=1}^s f_j \to \max,
    \end{equation}
where $f_{j} \geq 0$ for $j = 1, \dots, s$ and
  \begin{equation}\label{eq:DH}
  \sum_{j = 1}^{s} f_{j} \phi_{k}(d_{j}) \geq -\binom nk,\; k=0,\dots, n,
  \end{equation}
and $\phi_k$ are the Krawtchouk polynomials \eqref{eq:Krawtchouk}.
  
(b) Let $C \subset \cJ^{n, w}$ be a constant weight $s$-code with  distances $\sD=\{d_{1}, \dots, d_{s}\}$. Then
$|C|$ is bounded above by the value of linear program $\text{{LP}$_{J}$}(\sD)$ given by
  \begin{equation}\label{eq:LPJ}
 1+\sum_{j = 1}^{ s}f_{j} \to\max,
  \end{equation}
where $f_j\ge 0$ for $j=1,\dots,s$ and
  \begin{equation}\label{eq:DJ}
  \sum_{j = 1}^ s f_{j} \psi_{k}(d_{j}) \geq -1,\;  k=0,1,\dots,w,
  \end{equation}
and $\psi_k$ are the Hahn polynomials \eqref{eq:Hahn}.
\end{theorem}
Relations \eqref{eq:DH} and \eqref{eq:DJ} are called the {\em Delsarte inequalities}.

The {linear programming bounds} have been one of the main tools in estimating the value of $A(M,s)$ for $M=\cH_{2}^{n},\, \cJ^{n,w}$ as well as for spherical codes. Direct application of these bounds is difficult because of the need to exhaust
all possible subsets $\sD$; however it is possible to restrict the search by using the integrality
conditions quoted in the next theorem.

\begin{theorem}{\cite{MN11}} \label{thm:LRS_J}
Let $C \subset \cJ^{n, w}$ be an $s$-code with distances $\{d_{1}, \dots, d_{s}\}$. Let $N(\cJ^{n, w}, s) = \binom{n}{s-1}$ and 
\begin{equation*}
    K_{i} = \prod_{j \neq i} \frac{d_{j}}{d_{j}-d_{i}}\quad\text{for $i=1, \dots, s$.}
\end{equation*}
 If $|C|\geq 2N(\cJ^{n, w}, s)$, then
$K_{i}$ is an integer and 
\begin{equation*}
|K_{i}| \leq \left\lfloor \frac{1}{2} + \sqrt{\frac{N(\cJ^{n, w}, s)^{2}}{2N(\cJ^{n, w}, s)-2}+\frac{1}{4}} \right\rfloor \quad\text{for $i=1, \dots, s$.}
\end{equation*}
\end{theorem}

This statement forms an analog of a much earlier result for few-distance sets in $\reals^n$, namely the Larman-Rogers-Seidel theorem
\cite{LRS77}. A similar set of necessary conditions is also known for the Hamming space \cite{MN11}.

\subsection{Bounds for \texorpdfstring{$s$}--codes with small weight}
The following theorem is proved by considering a subset of inequalities in the LP bound. It allows us to calculate an upper bound of an $s$-code with fixed distance sets. To obtain an upper bound on $A(\cJ^{n, w}, s),$ we exhaust all possible $s$-tuples, aided by the conditions in Theorem~\ref{thm:LRS_J}.

\begin{theorem}\label{thm:Joh_lp_2}
Let $C \subset \cJ^{n, w}$ be a $2$-code with distances $d_1, d_2$. Suppose that for some $k_1,k_2\in\{1,\dots,w\}$
   \begin{align}
&\text{$\psi_{k_1}(d_1) \leq \psi_{k_1}(d_2)$ and $\psi_{k_2}(d_2) \leq \psi_{k_2}(d_1)$,}\label{eq:am1}\\
&\text{$\begin{vmatrix} \psi_{k_1}(d_1) & \psi_{k_2}(d_1) \\ \psi_{k_1}(d_2) & \psi_{k_2}(d_2) \end{vmatrix} > 0$,}\label{eq:am2}
\end{align}
then
\begin{equation}\label{eq:B2}
    |C| \leq \frac{\begin{vmatrix} 1 & 1 & 1 \\ 1 & \psi_{k_1}(d_1) & \psi_{k_2}(d_1) \\ 1 & \psi_{k_1}(d_2) & \psi_{k_2}(d_2) \end{vmatrix}}{\begin{vmatrix} \psi_{k_1}(d_1) & \psi_{k_2}(d_1) \\ \psi_{k_1}(d_2) & \psi_{k_2}(d_2) \end{vmatrix}}.
\end{equation}
\end{theorem}
\begin{proof}
Let $(f_1,f_2)$ be an optimal solution of Delsarte's linear program $\text{LP}_J$ \eqref{eq:LPJ}-\eqref{eq:DJ}, then $|C| \leq 1 + f_1 + f_2.$ The coefficients $f_1,f_2$ satisfy the Delsarte inequalities \eqref{eq:DJ}.
Consider the inequalities for a pair of degrees $k=k_1,k_2$. Adding them and using \eqref{eq:am1} we obtain
\begin{align*}
    0 &\geq (\psi_{k_2}(d_2) - \psi_{k_2}(d_1))(1 + \psi_{k_1}(d_1)f_1 + \psi_{k_1}(d_2)f_2) \\
    &\hspace*{.5in}+ (\psi_{k_1}(d_1) - \psi_{k_1}(d_2))(1 + \psi_{k_2}(d_1)f_1 + \psi_{k_2}(d_2)f_2) \\
    &=\begin{vmatrix} 0 & -1 & -1 \\ 1 & \psi_{k_1}(d_1) & \psi_{k_2}(d_1) \\ 1 & \psi_{k_1}(d_2) & \psi_{k_2}(d_2) \end{vmatrix} + \begin{vmatrix} \psi_{k_1}(d_1) & \psi_{k_1}(d_2) \\ \psi_{k_2}(d_1) & \psi_{k_2}(d_2) \end{vmatrix}(f_1+f_2)\\
    &=\begin{vmatrix} -1 & -1 & -1 \\ 1 & \psi_{k_1}(d_1) & \psi_{k_2}(d_1) \\ 1 & \psi_{k_1}(d_2) & \psi_{k_2}(d_2) \end{vmatrix}+
    \begin{vmatrix} \psi_{k_1}(d_1) & \psi_{k_1}(d_2) \\ \psi_{k_2}(d_1) & \psi_{k_2}(d_2) \end{vmatrix}(1+f_1+f_2).
\end{align*}
Now using \eqref{eq:am2} we obtain the claimed bound.
\end{proof}

This theorem enables us to prove a number of exact estimates of $A(\cJ^{n,w},2)$ for small $w.$ Before proceeding let us point out
a simple corollary for the case of $d_1=1,d_2=2.$ This is a known result, which follows from Erd{\H o}s-Ko-Rado (Thm.~\ref{EKR}) supplemented with Wilson's theorem; see Sec.~\ref{sec:forbidden} for details and references.
\begin{corollary}\label{cor:1-2}
 If $n\ge 3w-3$  then
   \begin{equation}\label{eq:3w}
A(\cJ^{n,w}, \{1, 2\})=\binom{n-w+2}{2}.
   \end{equation}
\end{corollary}

\begin{proof}
Take in Theorem \ref{thm:Joh_lp_2} $k_1=w,k_2=w-1.$ From \eqref{eq:Hahn} we find
  \begin{gather*}
   \psi_w(1)=-\frac 1{n-w},\;\psi_w(2)=\frac 2{(n-w)(n-w-1)}\\
   \psi_{w-1}(1)=\frac{n-3w+2}{w(n-w)}, \; \psi_{w-1}(2)=-\frac{2(2n-5w+4)}{w(n-w)(n-w-1)}.  
  \end{gather*}
Then $\psi_{w}(2)-\psi_w(1)=\frac{n-w+1}{(n-w)(n-w-1)}\ge 0$ and $\psi_{w-1}(1)-\psi_w(2)=\frac{(n-3 w+3) (n-w+2)}{w (n-w-1) (n-w)}\ge 0,$
so conditions \eqref{eq:am1} are satisfied. Further, 
   $$
   \psi_w(1)\psi_{w-1}(2)-\psi_w(2)\psi_{w-1}(1)=\frac{2 (n-2 w+2)}{w (n-w-1) (n-w)^2}>0,
   $$
so bound \eqref{eq:B2} applies. Evaluating it, we find that $A(\cJ^{n,w}, \{1, 2\})$
is at most the right-hand side of \eqref{eq:3w}. Finally, Proposition \ref{prop:MN_J2} implies that \eqref{eq:3w} holds with equality.
  \end{proof}
It is interesting that the positivity conditions \eqref{eq:am1}, \eqref{eq:am2} recover the exact bound for $n$ in Wilson's theorem by a seemingly different argument.

Similar considerations enable one to recover some other known bounds for $A(\cJ^{n,w},\sD)$. For instance, for $\sD=\{w-1,w\}$ we can take $k_1=2,k_2=1$ and check that \eqref{eq:am1}, \eqref{eq:am2} are satisfied. This yields $A(\cJ^{n,w},\{w-1,w\})\le n(n-1)/w(w-1),$ which coincides with the Johnson bound. Similarly, the inequality $A(\cJ^{n,w},\{1,2,3\})\le \binom{n-w+3}3$ can be proved by taking $k_1=w,k_2=w-1,k_3=w-2$ and recovering the EKR upper bound for all $n\ge 4w-8$,
which is Wilson's condition. 

In the next theorem we use bound \eqref{eq:B2} to estimate the maximum size of constant weight codes with 2 distances irrespective of the values of the distances; however we will have to assume that the weight $w$ is small.

\begin{theorem}\label{thm:Johnson2}\hfill\\
(a) $A(\cJ^{n, 4}, 2) = \binom{n-2}{2}$ for $n \geq 9$.

\nd(b) $A(\cJ^{n, 5}, 2) = \binom{n-3}{2}$ for $n \geq 12$.

\nd(c) $A(\cJ^{n, 6}, 2) = \binom{n-4}{2}$ for $n \geq 35$.
\end{theorem}

\begin{proof} We will use Theorem~\ref{thm:Joh_lp_2} with $k_i = w+1-d_i, i=1,2$. The scheme of the proof is the same for each of the three cases. We illustrate it in detail for Part (a).

(a) Theorem  \ref{thm:LRS_J} restricts the possible values of distance pairs $d_1,d_2$. Namely, 
since $2N(\cJ^{n, 4}, 2) = 2n \leq \binom{n-2}{2}$ for $n \geq 9$, we only need to consider the 
pairs $\{1, 2\}$, $\{2, 3\}$, $\{2, 4\}$, and $\{3, 4\}$. Now let us use Theorem~\ref{thm:Joh_lp_2}. For $\sD=\{1,2\}$ we
have $k_1=4,k_2=3.$ Denote $E_1=\psi_{k_2}(d_1)-\psi_{k_2}(d_2), E_2=\psi_{k_1}(d_2)-\psi_{k_1}(d_1), E_3=\begin{vmatrix} \psi_{k_1}(d_1) & \psi_{k_1}(d_2) \\ \psi_{k_2}(d_1) & \psi_{k_2}(d_2) \end{vmatrix}$, then we obtain
  \begin{gather*}
  E_1=\frac{(n-2)(n-9)}{4(n-4)(n-5)}, \;E_2=\frac{n-3}{(n-4)(n-5)},\;
  E_3=\frac{n-6}{2(n-4)^2(n-5)}.
  \end{gather*}

These fractions are nonnegative for all $n\ge 9,$ and thus from \eqref{eq:B2} we get
  $$
  |C| \leq \frac{(n-2)(n-3)}{2} = \binom{n-2}{2}.
  $$
The remaining three cases are checked in a similar way. The obtained expressions are listed in Appendix~\ref{App:4,2}. Altogether
this argument proves that for $n\ge 12$ the value of $A(\cJ^{n, 4}, 2)\le \binom{n-2}2.$ Together with the construction of Proposition \ref{prop:MN_J2} and Proposition \ref{prop:2_dis_equal_J}, we obtain the desired result.

(b) Since $2N(\cJ^{n, 5}, 2) = 2n \leq \binom{n-3}{2}$ for $n \geq 10$, by Theorem \ref{thm:LRS_J}, we only need to consider the distance pairs $\{1, 2\}$, $\{2, 3\}$, $\{2, 4\}$, $\{3, 4\}$, $\{4, 5\}$. The expressions listed in Appendix~\ref{App:5,2}
show that $A(\cJ^{n, 5}, 2)\le \binom{n-3}{2}$ for $n\ge 18$. Along with the construction given in Proposition \ref{prop:MN_J2} and Proposition \ref{prop:2_dis_equal_J}, we obtain the desired result.

(c) Since $2N(\cJ^{n, 6}, 2) = 2n \leq \binom{n-4}{2}$ for $n \geq 12$, by Theorem \ref{thm:LRS_J}, we only need to consider the distance pairs $\{1, 2\}$, $\{2, 3\}$, $\{2, 4\}$, $\{3, 4\}$, $\{3, 6\}$, $\{4, 5\}$, $\{4, 6\}$, $\{5, 6\}$.
The expressions listed in Appendix~\ref{App:6,2}
show that $A(\cJ^{n, 6}, 2)\le \binom{n-4}{2}$ for $n\ge 35$. Along with the construction of Proposition \ref{prop:MN_J2}
we obtain the desired result.
\end{proof}

Theorem~\ref{thm:Joh_lp_2} affords a generalization for the case of $s$-codes with $s\ge 2.$
\begin{theorem}\label{thm:Joh_lp_s}
Let $C \subset \cJ^{n, w}$ be a $s$-code with distance $\{d_1, \dots, d_s\}$. Let $k_1, \dots, k_s \in \{1, \dots, w\}$ and 
$\bar\psi_{k} = (\psi_{k}(d_1), \dots, \psi_{k}(d_s))^{\top}$ for $k = 0, \dots, w$. Suppose that
\begin{gather}
    \text{$(-1)^{s}\; \begin{array}{|*7{@{\hspace{.05in}}c}|} \bar\psi_{k_{1}}  &\dots  &\bar\psi_{k_{j-1}} & \bar\psi_{0} &\bar\psi_{k_{j+1}} & \dots  &\bar\psi_{k_s} \end{array} \leq 0$ for $j = 1, \dots, s$,}
    \label{eq:am1s}\\[.1in]
    \text{ $(-1)^{s+1}\begin{vmatrix} \bar\psi_{k_1} & \dots & \bar\psi_{k_s} \end{vmatrix} < 0$.}\label{eq:am2s}
\end{gather}
Then
\begin{equation}
\label{eq:Bs}
    |C| \leq \frac{\begin{vmatrix} 1 & 1 & \dots & 1 \\ 
    {\mathbf 1} & \bar\psi_{k_1} & \dots & \bar\psi_{k_s}  \end{vmatrix}}
    {\begin{vmatrix}\bar\psi_{k_1}&\dots&\bar\psi_{k_s}\end{vmatrix}}.
\end{equation}
\end{theorem}
\begin{proof}
Let $(f_1, \dots, f_s)$ be an optimal solution of Delsarte's linear program $\text{LP}_J$ \eqref{eq:LPJ}-\eqref{eq:DJ}, 
then $|C| \leq 1 + f_1 +\dots+ f_s.$ By the Delsarte inequalities \eqref{eq:DJ},
\begin{equation*}
    1+ \sum_{i=1}^{s}\psi_{k}(d_i)f_i \geq 0, k = 0, \dots, w.
\end{equation*}
We will use a subset of these inequalities for the degrees $k_1,\dots,k_s$. On account of \eqref{eq:am1s} we obtain
\begin{align*}
    0 \geq& (-1)^{s}\sum_{j=1}^{s}
         \begin{array}{|*7{@{\hspace{.05in}}c}@{\hspace*{0in}}|} \bar\psi_{1} & \dots & \bar\psi_{k_{j-1}} & \bar\psi_{0} & \bar\psi_{k_{j+1}} & \dots & \bar\psi_{k_s} 
         \end{array} 
         \;\Big(1 + \sum_{t=1}^{s}\psi_{k_{j}}(d_{t})f_{t} \Big)\\
    =& (-1)^{s}\begin{vmatrix}
    0 & -1 & \dots & -1\\
    \bar\psi_{0} & \bar\psi_{k_{1}} & \dots & \bar\psi_{k_{s}}
    \end{vmatrix} + 
    (-1)^{s}\begin{vmatrix}
   \bar\psi_{k_{1}} & \dots & \bar\psi_{k_{s}}
    \end{vmatrix}\sum_{t=1}^{s} f_{t}.
\end{align*}
Now recalling that $\psi_0\equiv 1$ and using assumption \eqref{eq:am2s} we obtain
   $$
   1+f_1+\dots+f_s\le \frac{\begin{vmatrix} 1 & 1 & \dots & 1 \\ 
    {\mathbf 1} & \bar\psi_{k_1} & \dots & \bar\psi_{k_s}  \end{vmatrix}}
    {\begin{vmatrix}\bar\psi_{k_1}&\dots&\bar\psi_{k_s}\end{vmatrix}}.
    $$
\end{proof}

In the next two statements we examine a number of implications of this theorem. For the remainder of this section let $k_i=w+1-d_i$
for all $i$ as appropriate. 
\begin{theorem}\label{thm:Johnson3}\hfill\\
(a) $A(\cJ^{n, 4}, 3) = \binom{n-1}{3}$ for $n \geq 11$.

\nd(b) $A(\cJ^{n, 5}, 3) = \binom{n-2}{3}$ for $n \geq 12$.

\nd(b) $A(\cJ^{n, 6}, 3) = \binom{n-3}{3}$ for $n \geq 16$.

\nd(c) $A(\cJ^{n, 7}, 3) = \binom{n-4}{3}$ for $n \geq 20$.
\end{theorem}
Note that Corollary \ref{cor:ww-1} implies Part (a), with a stronger bound $n\ge 6.$

\begin{proof} We follow the pattern of Theorem~\ref{thm:Johnson2}, supplying details for a part of the proof and moving
the rest to the Appendix.

(a) Begin by noting that Theorem  \ref{thm:LRS_J} restricts the possible values of $d_1,d_2,d_3$. Namely, 
since $2N(\cJ^{n, 4}, 3) = 2\binom{n}{2} \leq \binom{n-1}{3}$ for $n \geq 11$, 
we only need to consider the distance sets $\{1, 2, 3\}$, $\{1, 3, 4\}$, $\{2, 3, 4\}$. 
Now let us use Theorem \ref{thm:Joh_lp_s}. For $\sD=\{1, 2, 3\}$ we
have $k_1=4,k_2=3,k_3=2.$ Let 
   \begin{gather*}
   E_1=-\begin{array}{|*3{@{\hspace{.05in}}c}@{\hspace*{0in}}|} \bar\psi_{0}  & \bar\psi_{k_{2}} & \bar\psi_{k_{3}}\end{array}\,,\; 
   E_2=-\begin{array}{|*3{@{\hspace{.05in}}c}@{\hspace*{0in}}|} \bar\psi_{k_{1}}  & \bar\psi_{0} & \bar\psi_{k_{3}}\end{array}\,,\\
    E_3=-\begin{array}{|*3{@{\hspace{.05in}}c}@{\hspace*{0in}}|} \bar\psi_{k_{1}}  & \bar\psi_{k_{2}} & \bar\psi_{0}\end{array}\,,\; 
    E_4=\begin{array}{|*3{@{\hspace{.05in}}c}@{\hspace*{0in}}|} \bar\psi_{k_1} & \bar\psi_{k_{2}} & \bar\psi_{k_{3}}\end{array}\,.
   \end{gather*}
We compute
\begin{gather*}
  E_1=-\frac{(n-1)(n-2)\left(n^2-14n+51\right)}{24(n-4)(n-5)^2(n-6)}\,, \;E_2=-\frac{(n-1)(n-3)(n-8)}{6(n-4)^2(n-5)(n-6)}\,,\\
  \;E_3=-\frac{(n-2)(n-3)}{2(n-4)^2(n-5)^2},\;
  E_4=-\frac{1}{4(n-4)^{2}(n-5)}. \\
\end{gather*}
These fractions are nonpositive for all $n\ge 8,$ and thus from \eqref{eq:Bs} we get
  $$
  |C| \leq \frac{(n-1)(n-2)(n-3)}{6} = \binom{n-1}{3}.
  $$
The remaining two cases are checked in a similar way. The obtained expressions are listed in Appendix~\ref{App:4,3}. Altogether
this argument proves that for $n\ge 11$ the value of $A(\cJ^{n, 4}, 3)\le \binom{n-1}{3}.$ Together with the construction of Proposition \ref{prop:MN_J2}, we obtain the desired result.

(b) Since $2N(\cJ^{n, 5}, 3) = 2\binom{n}{2} \leq \binom{n-2}{3}$ for $n \geq 13$, by Theorem \ref{thm:LRS_J}, we only need to consider the distance triples $\{1, 2, 3\}$, $\{1, 3, 4\}$, $\{2, 3, 4\}$, $\{2, 3, 5\}$, $\{3, 4, 5\}$. The expressions listed in Appendix~\ref{App:5,3}
show that $A(\cJ^{n, 5}, 3)\le \binom{n-2}{3}$ for $n\ge 21$. Along with the construction in Proposition \ref{prop:MN_J2} and Proposition \ref{prop:3_dis_equal_J}, we obtain the desired result.

(c) Since $2N(\cJ^{n, 6}, 3) = 2\binom{n}{2} \leq \binom{n-3}{3}$ for $n \geq 15$, by Theorem \ref{thm:LRS_J}, we only need to consider the distance triples $\{1, 2, 3\}$, $\{1, 3, 4\}$, $\{2, 3, 4\}$, $\{2, 3, 5\}$, $\{2, 4, 6\}$ $\{3, 4, 5\}$, $\{3, 4, 6\}$, $\{3, 5, 6\}$, $\{4, 5, 6\}$. The expressions listed in Appendix~\ref{App:6,3}
show that $A(\cJ^{n, 6}, 3)\le \binom{n-3}{3}$ for $n\ge 31$. Now Propositions \ref{prop:MN_J2} and  \ref{prop:3_dis_equal_J} imply the desired result.

(d) Since $2N(\cJ^{n, 7}, 3) = 2\binom{n}{2} \leq \binom{n-4}{3}$ for $n \geq 17$, by Theorem \ref{thm:LRS_J}, we only need to consider the distance triples $\{1, 2, 3\}$, $\{1, 3, 4\}$, $\{2, 3, 4\}$, $\{2, 3, 5\}$, $\{2, 4, 6\}$ $\{3, 4, 5\}$, $\{3, 4, 6\}$, $\{3, 4, 7\}$ $\{3, 5, 6\}$, $\{4, 5, 6\}$, $\{4, 6, 7\}$, $\{5, 6, 7\}$. The expressions listed in Appendix~\ref{App:7,3}
show that $A(\cJ^{n, 7}, 3)\le \binom{n-4}{3}$ for $n\ge 50$. As above, Propositions \ref{prop:MN_J2} and  \ref{prop:3_dis_equal_J} imply the desired result.
\end{proof}

\begin{theorem}\label{thm:Johnson4}\hfill\\
(a) $A(\cJ^{n, 5}, 4) = \binom{n-1}{4}$ for $n \geq 15$.

\nd(b) $A(\cJ^{n, 6}, 4) = \binom{n-2}{4}$ for $n \geq 15$.

\nd(c) $A(\cJ^{n, 7}, 4) = \binom{n-3}{4}$ for $n \geq 24$.
\end{theorem}
Note that Corollary \ref{cor:ww-1} implies Part (a), with a stronger bound $n\ge 8.$
\begin{proof} We again follow the pattern of Theorem~\ref{thm:Johnson2} in writing the proof.

(a) Since $2N(\cJ^{n, 5}, 4) = 2\binom{n}{3} \leq \binom{n-1}{4}$ for $n \geq 15$, by Theorem \ref{thm:LRS_J}, we only need to consider the distance tuples $\{1, 2, 3, 4\}$, $\{2, 3, 4, 5\}$. Now let us use Theorem \ref{thm:Joh_lp_s}. For $\sD=\{1, 2, 3, 4\}$ we
have $k_1=5,k_2=4,k_3=3,k_4=2.$ Denote 
   \begin{gather*}
  E_1=\begin{array}{|*4{@{\hspace{.05in}}c}@{\hspace*{0in}}|} \bar\psi_{0}  & \bar\psi_{k_{2}} & \bar\psi_{k_{3}} & \bar\psi_{k_{4}}\end{array}\,, \;
  E_2=\begin{array}{|*4{@{\hspace{.05in}}c}@{\hspace*{0in}}|} \bar\psi_{k_{1}}  & \bar\psi_{0} & \bar\psi_{k_{3}} & \bar\psi_{k_{4}}\end{array}\,, \;
  E_3=\begin{array}{|*4{@{\hspace{.05in}}c}@{\hspace*{0in}}|} \bar\psi_{k_{1}}  & \bar\psi_{k_{2}} & \bar\psi_{0} & \bar\psi_{k_{4}}\end{array}\,, \\
  E_4=\begin{array}{|*4{@{\hspace{.05in}}c}@{\hspace*{0in}}|} \bar\psi_{k_1} & \bar\psi_{k_{2}} & \bar\psi_{k_{3}} & \bar\psi_{0} \end{array}\,, \;
  E_5=-\begin{array}{|*4{@{\hspace{.05in}}c}@{\hspace*{0in}}|} \bar\psi_{k_1} & \bar\psi_{k_{2}} & \bar\psi_{k_{3}} & \bar\psi_{k_{4}} \end{array}\,,
  \end{gather*} 
  then we obtain
\begin{gather*}
E_1=-\frac{(n-1)(n-2)(n-3)(n-4)(n-10)\left(n^2-15n+60\right)}{500(n-5)^2(n-6)^2(n-7)^2(n-8)},\\ E_2=-\frac{(n-1)(n-2)(n-4)^2\left(n^2-17n+78\right)}{100(n-5)^3(n-6)^2(n-7)(n-8)},\\ 
E_3=-\frac{(n-1)(n-3)(n-4)(n-10)}{25(n-5)^2(n-6)^2(n-7)^2},
E_4=-\frac{3(n-2)(n-3)(n-4)}{25(n-5)^3(n-6)^2(n-7)},\\
E_5=-\frac{6(n-4)}{125(n-5)^3(n-6)(n-7)}.\\
\end{gather*}
These fractions are nonpositive for all $n\ge 10,$ and thus from \eqref{eq:Bs} we get
  $$
  |C| \leq \frac{(n-1)(n-2)(n-3)(n-4)}{24} = \binom{n-1}{4}.
  $$
The remaining tuple $\{2, 3, 4, 5\}$ is checked in a similar way. The obtained expressions are listed Appendix~\ref{App:5,4}. Altogether
this argument proves that for $n\ge 15$ the value of $A(\cJ^{n, 5}, 4)\le \binom{n-1}{4}.$ Together with the construction of Proposition  \ref{prop:MN_J2}, we obtain the desired result.

(b) Since $2N(\cJ^{n, 6}, 4) = 2\binom{n}{3} \leq \binom{n-2}{4}$ for $n \geq 17$, by Theorem \ref{thm:LRS_J}, we only need to consider the distance sets $\{1, 2, 3, 4\}$, $\{1, 4, 5, 6\}$, $\{2, 3, 4, 5\}$, $\{2, 3, 4, 6\}$, $\{2, 4, 5, 6\}$, $\{3, 4, 5, 6\}$. The expressions listed in Appendix~\ref{App:6,4}
show that $A(\cJ^{n, 6}, 4)\le \binom{n-2}{4}$ for $n\ge 30$. Now Proposition  \ref{prop:MN_J2} and Proposition \ref{prop:4_dis_equal_J} imply the desired result.

(c) Since $2N(\cJ^{n, 7}, 4) = 2\binom{n}{3} \leq \binom{n-3}{4}$ for $n \geq 20$, by Theorem \ref{thm:LRS_J}, we only need to consider the distance sets $\{1, 2, 3, 4\}$, $\{1, 4, 5, 6\}$, $\{2, 3, 4, 5\}$, $\{2, 3, 4, 6\}$, $\{2, 4, 5, 6\}$, $\{3, 4, 5, 6\}$, $\{3, 4, 6, 7\}$, $\{4, 5, 6, 7\}$. The expressions listed in Appendix~\ref{App:7,4}
show that $A(\cJ^{n, 7}, 4)\le \binom{n-3}{4}$ for $n\ge 45$. As above, Proposition  \ref{prop:MN_J2} and Proposition \ref{prop:4_dis_equal_J} imply the desired result.
\end{proof}

Numerical results obtained in Sec.~\ref{sec:numerical2} relying on the SDP method imply that the approach
considered in this section generally is not strong enough to tighten the gap between the lower and upper bounds. 
The limited scope of Theorem \ref{thm:Joh_lp_s} is highlighted already by the case of $w=7,s=2,$ where the obtained
upper bound falls short of reaching $\binom{n-2}{2}$ for large $n$. See Remark~\ref{rmk:72} for the details.

\subsection{Upper bounds on \texorpdfstring{$2$}--codes for large \texorpdfstring{$n$}.}\label{sec:large-n}
From the preceding results it is clear that for large $n$ the size of maximum 2-codes in the Johnson space $\cJ^{n,w}$ is proportional
to $n^2,$  namely $A(\cJ^{n,w},2) \le cn^2+o(n^2).$ In this section we make the first term of the asymptotics more precise by
evaluating the behavior of the bound in Theorem~\ref{thm:Joh_lp_2}. As a result, we will compute the coefficient $c=c(w,d_1,d_2)$.

We need some asymptotic estimates for the Hahn polynomials. More precise results are abundant in the literature (e.g., 
\cite{LW13}); however the simple bound that stated below is easier to use to compute bounds on codes.
\begin{proposition}\label{prop: Hahn}
Let $k,x\in\{1,2,\dots,w\}.$ As $n \to \infty$, the Hahn polynomial has the order $\psi_k(x) = \frac{1}{\binom{w}{x}}(A + Bn^{-1} + Cn^{-2}) + O(n^{-3})$, where
\begin{align*}
    A &= \binom{w-k}{x}, \\
    B &= -kx\binom{w-k+1}{x}, \\
    C &= -kxw\binom{w-k+1}{x} + \frac{k(k-1)x(x-1)}{2}\binom{w-k+2}{x}.
\end{align*}
\end{proposition}

The proof of this proposition appears in the end of this section.

\begin{corollary}
As $n \to \infty$,
 \begin{align*}
 &\psi_k(x) \sim \frac{\binom{w-k}{x}}{\binom{w}{x}}\text{ if }x \leq w-k\\
 & \psi_k(x) \sim \frac{-kx}{\binom{w}{x}} \cdot n^{-1}\text{ if }x = w-k+1\\
  & \psi_k(x) \sim \frac{k(k-1)x(x-1)}{2\binom{w}{x}} \cdot n^{-2}\text{ if }x = w-k+2\\
 & \psi_k(x) = O(n^{-3})\text{ if }x \geq w-k+3.\\
\end{align*}
\end{corollary}

The main result of this section is stated next.
\begin{theorem}\label{thm:asymptotic_2-dist}
As $n \to \infty$, $
   A(\cJ^{n,w},\{d_1,d_2\})\le c(w,d_1,d_2)n^2(1+o(1)),$ where
  \begin{equation}\label{eq:cwd}
     c(w,d_1,d_2)=\begin{cases}
        \displaystyle \frac{2\binom{w}{d_2}}{d_1d_2(w+1-d_1)(w+1-d_2)}  &\text{if }d_1=d_2-1\\[.1in]
         \displaystyle  \frac{\binom{d_2-1}{d_1}\binom{w}{d_2}}{d_1d_2(w+1-d_1)(w+1-d_2)}&\text{if }d_1<d_2-1.
          \end{cases}
   \end{equation}
\end{theorem}
\begin{proof}
First we analyze the case of $d_1 + 1 = d_2.$ We have $\psi_{k_1}(d_1) \sim \frac{-k_1d_1}{\binom{w}{d_1}} \cdot n^{-1}$, $\psi_{k_1}(d_2) \sim \frac{k_1k_2d_1d_2}{2\binom{w}{d_2}} \cdot n^{-2}$, $\psi_{k_2}(d_1) \sim \frac{1}{\binom{w}{d_1}}$, $\psi_{k_2}(d_2) \sim \frac{-k_2d_2}{\binom{w}{d_2}} \cdot n^{-1}$. For sufficiently large $n$, assumptions \eqref{eq:am1} of Theorem \ref{thm:Joh_lp_2} holds. The order of the $2$ by $2$ determinant is
    \begin{equation*}
        \begin{vmatrix} \psi_{k_1}(d_1) & \psi_{k_2}(d_1) \\ \psi_{k_1}(d_2) & \psi_{k_2}(d_2) \end{vmatrix} \sim \begin{vmatrix}
            \frac{-k_1d_1}{\binom{w}{d_1}} \cdot n^{-1} & \frac{1}{\binom{w}{d_1}} \\
            \frac{k_1k_2d_1d_2}{2\binom{w}{d_2}} \cdot n^{-2} & \frac{-k_2d_2}{\binom{w}{d_2}} \cdot n^{-1}
        \end{vmatrix} = \frac{k_1k_2d_1d_2}{2\binom{w}{d_1}\binom{w}{d_2}} \cdot n^{-2},
    \end{equation*}
    which is positive for sufficiently large $n$, so we can apply the bound \eqref{eq:B2}. The order of the $3$ by $3$ determinant is
    \begin{equation*}
        \begin{vmatrix} 1 & 1 & 1 \\ 1 & \psi_{k_1}(d_1) & \psi_{k_2}(d_1) \\ 1 & \psi_{k_1}(d_2) & \psi_{k_2}(d_2) \end{vmatrix} \sim \begin{vmatrix} 1 & 1 & 1 \\ 1 & \Theta(n^{-1}) & \frac{1}{\binom{w}{d_1}} \\ 1 & \Theta(n^{-2}) & \Theta(n^{-1}) \end{vmatrix} = \frac{1}{\binom{w}{d_1}} (1+o(1)),
    \end{equation*}
    so the order of the upper bound is $\frac{2\binom{w}{d_2}}{d_1d_2(w+1-d_1)(w+1-d_2)} \cdot n^2$.

Now suppose that $d_1 + 1 < d_2,$  then the order of $\psi_{k_1}(d_1)$ and $\psi_{k_2}(d_2)$ are the same as above, $\psi_{k_2}(d_1) \sim \frac{\binom{d_2-1}{d_1}}{\binom{w}{d_1}}$, and $\psi_{k_1}(d_2) = O(n^{-3})$.  Similar calculations show that assumptions \eqref{eq:am1} and \eqref{eq:am2} of Theorem \ref{thm:Joh_lp_2} hold, and the bound \eqref{eq:B2} gives the order $\frac{\binom{d_2-1}{d_1}\binom{w}{d_2}}{d_1d_2(w+1-d_1)(w+1-d_2)} \cdot n^2$.
\end{proof}

\begin{remark}
The bound of Theorem \ref{thm:asymptotic_2-dist} is tight in two notable cases. The case $d_1=1$ and $d_2=2$, where the correct constant is $c=\frac 12$, is already covered by Corollary \ref{cor:1-2}. In the case $d_1=w-1$ and $d_2=w$, Theorem \ref{thm:asymptotic_2-dist} implies that $c\leq \frac 1 {w(w-1)}$. On the one hand, this coincides with the Johnson bound \cite{J62} already mentioned above. On the other hand, the existence of 2-distance sets of this size for sufficiently large $n$ immediately follows from the celebrated work of Wilson on the existence of balanced incomplete block designs \cite{W75}.
\end{remark}

{
\begin{remark}\label{rmk:72}
As mentioned above, Theorem~\ref{thm:Joh_lp_s} is not strong enough to determine the bound on $s$-codes in general, with an example that Theorem~\ref{thm:Joh_lp_s} cannot give an upper bound $\binom{n-2}{2}$ when $w=7, s=2$. Indeed, consider $d_1=2$, $d_2=4$ in this case. This pair cannot be ruled out by Theorem~\ref{thm:LRS_J}, while Theorem~\ref{thm:Joh_lp_2} can only give an upper bound of $c(7, 2, 4) \cdot n^2(1+o(1)) = \frac{35}{64}n^2(1+o(1))$ asymptotically. In general, our methods cannot deal with any pairs $(d_1, d_2)$ such that $d_2-d_1$ divides $d_1$ and $c(w, d_1, d_2) \geq \frac{1}{2}$. 
\end{remark}
}

\noindent{\em Proof of Proposition \ref{prop: Hahn}}
Recall the expression for $\psi_k(x)$ in \eqref{eq:Hahn}.
We begin with finding the order of $\frac{\binom{n+1-k}{j}}{\binom{n-w}{j}}$. By taking logarithms, we have
\begin{align*}
    \log\frac{\binom{n+1-k}{j}}{\binom{n-w}{j}} &= \log\frac{(n+1-k) \dots (n+2-k-j)}{(n-w) \dots (n-w-j+1)} \\
    &= \sum_{r=0}^{j-1} \left(\log\left(1-\frac{k+r-1}{n}\right) - \log\left(1-\frac{w+r}{n}\right)\right) \\
    &= \sum_{r=0}^{j-1}\left( \frac{(w+r) - (k+r-1)}{n} + \frac{(w+r)^2 - (k+r-1)^2}{2n^2}\right) + O(n^{-3}) \\
    &= j(w-k+1) \cdot n^{-1} + \frac{j(w-k+1)(j+w+k-2)}{2} \cdot n^{-2} + O(n^{-3}).
\end{align*}
Expanding $\exp\Big({\log\frac{\binom{n+1-k}{j}}{\binom{n-w}{j}}}\Big)$ into a power series, we obtain
\begin{align*}
    \frac{\binom{n+1-k}{j}}{\binom{n-w}{j}} &= 1 + j(w-k+1) \cdot n^{-1} + \frac{j(w-k+1)(j+w+k-2) + j^2(w-k+1)^2}{2} \cdot n^{-2}\\
    &\quad + O(n^{-3}) \\
    &= 1 + j(w-k+1) \cdot n^{-1} + jw(w-k+1)n^{-2} \\
    &\quad + \frac{j(j-1)(w-k+1)(w-k+2)}{2} \cdot n^{-2} + O(n^{-3}),
\end{align*}
and
\begin{equation}
    \psi_k(x) = \alpha + (w-k+1)\beta n^{-1} + w(w-k+1)\beta n^{-2} + \binom{w-k+2}{2}\gamma n^{-2} + O(n^{-3}),
\end{equation}
where
   \begin{align*}
   \alpha&=\sum_{j=0}^{k} (-1)^j \binom{k}{j}\frac{\binom{x}{j}}{\binom{w}{j}}=
   \frac1{\binom wx}\sum_{j=0}^{k} (-1)^j \binom kj\binom{w-j}{w-x}=\frac{\binom{w-k}{x}}{\binom{w}{x}}\\
   \beta&=\sum_{j=0}^{k} (-1)^j j\binom{k}{j}\frac{\binom{x}{j}}{\binom{w}{j}}=
     \frac k{\binom wx} \sum_{j=0}^{k} (-1)^j \binom{k-1}{j-1}\binom{w-j}{w-x}=-k\frac{\binom{w-k}{x-1}}{\binom{w}{x}}\\
   \gamma&= \sum_{j=0}^{k} (-1)^j j(j-1)\binom{k}{j}\frac{\binom{x}{j}}{\binom{w}{j}}=
     \frac {k(k-1)}{\binom wx} \sum_{j=0}^{k} (-1)^j \binom{k-2}{j-2}\binom{w-j}{w-x}\\
      &=k(k-1)\frac{\binom{w-k}{x-2}}{\binom{w}{x}},
  \end{align*}
where the last step in each of these three equalities is a version of the Vandermonde convolution.

\section{Numerical results} \label{sec:numerical}

In this section we list some new numerical results for the size of $s$-codes with small $s$ together with the previously
known results of \cite{BM11,MN11}. While many of them for large values of $n$ are superseded by the general results listed in
Table~\ref{table:gb}, for small $n$ the methods used in the previous section are not strong enough to yield exact values.
Here the numerical results are useful in that they enable us to decrease the values on the code length $n$ starting with which
the general bounds apply. Further, the numerical results also extend to values of $w$ that could not be handled by the approach of Section \ref{sec:small-w}.  The results cited below combine the calculations performed in 
\cite{BM11,MN11} with several new sets of parameters found here. The new sets were obtained by using
either LP or Schrijver's semidefinite programming bounds (for reference, their statements are listed in
Appendix \ref{sec:SDPbounds}; see \cite{S05}  for a complete treatment).
 
\subsection{Bounds for 3-codes in the Hamming space}
Unlike the case of $s=2$, for 3 distances we are able to obtain exact values of $A(\cH_2^n,3)$ only for a set of small values of 
$n$. Namely, the following is true.
\begin{proposition} For $8\le n\le 22, 24\le n\le 37$ and $n=44$, $A(\cH_2^n,3)=n+\binom n3.$
\end{proposition}
Most of these results were established in \cite{BM11} and \cite{MN11}, and our contribution is the values $n=34,35.$
Calculations are simplified by the observation that the triples of distances in a 3-code satisfy certain integrality conditions of LRS type similar to those listed in Theorem~\ref{thm:LRS_J} for the Johnson space \cite{MN11}. Using them, we computed the upper bounds by implementing the semidefinite program of Theorem \ref{thm:sdp_Ham} and performing the calculations for the allowable triples. 

\subsection{Bounds for \texorpdfstring{$s$}--codes in the Johnson space}
\label{sec:numerical2}
Previously known explicit bounds for $A(\cJ^{n,w},2)$, cited from \cite{BM11,MN11}, are as follows. Note that the results for $w=s+1$ are now fully covered by our Corollary~\ref{cor:ww-1}. The other cited results either augment our bounds for small $n$, or simply remain the best known where our methods are not powerful enough.
\begin{proposition}
\label{prop:2_dis_equal_J}
$A(\cJ^{n, w}, 2) = \binom{n-w+2}{2}$ if $n$ and $w$ satisfy one of the following conditions:
{\begin{enumerate}
    \item
    $w=3$ and $6\le n\le 46$;
   \item
   $w=4$ and $9\le n\le 46$;
    \item
    $w=5$ and $12\le n\le 46$;
    \item
    $w=6$ and $15\le n\le24$ or $35 \leq n \leq 46$ .
\end{enumerate}}
\end{proposition}
These results supplement the LP proof of Theorem~\ref{thm:Johnson2} by lowering the smallest value of $n$
for which the claims in the theorem apply. For $w=4,5,6$ the added lengths are $9\le n\le 11, 12\le n\le 17,$ and
$15\le n\le 24$, respectively.

Similarly, for $3$-codes we have the following statement due to \cite{MN11} except for the case $(n, w) = (12, 5)$ which is new.

\begin{proposition}
\label{prop:3_dis_equal_J}
$A(\cJ^{n, w}, 3) = \binom{n-w+3}{3}$ if $n$ and $w$ satisfy one of the following conditions:
\begin{enumerate}
    \item
    $w=4$ and $11 \leq n \leq {58}$;
    \item
    $w=5$ and $12 \leq n \leq {58}$;
    \item
    $w=6$ and $16\le n\le {58}$;
    \item
    $w=7$ and $20\le n\le {58}$;
    \item
    $w=8$ and $25 \leq n \leq {58}$.
\end{enumerate}
\end{proposition}
As above, Part (a) is completely covered by Cor~\ref{cor:ww-1}. 
For $w=5,6,7$ these results add a number of small values of $n$ to the LP proofs of Theorem ~\ref{thm:Johnson3}: namely, the added lengths are $12\le n\le 20, 16\le n\le 30$, and $20\le n\le 49$, respectively. For $w=8$ the
LP argument does not yield any results, so part (4) above remains the state of the art.

To obtain the result for the new pair $(12,5)$ we sharpened the calculation of \cite{MN11} as follows. Recall that the generalized LRS theorem (Theorem~\ref{thm:LRS_J}) gives necessary integrality conditions for the distances of an $s$-code of sufficiently large size.  For the distances that satisfy these conditions, the authors of \cite{MN11} used the LP and harmonic bounds, arriving at their results. At the same time, for the other cases, i.e., when the conditions are
not satisfied, they used a trivial upper bound $2N(\cJ^{n,w},s)-1$ on the size of the code. We examined all the possible distance sets when the  $2N(\cJ^{n,w},s)-1$ is larger than $\binom{n-w+s}{s}$ in~Proposition~\ref{prop:MN_J2} and managed to rule them out using the LP bound. Moreover, we observed that the LP bound
applies for $n\le 58$ while \cite{MN11} has $n=50$ as the upper limit in their
statement.

Using the same approach for $s=4$, we managed to add two new sets of parameters $(n, w) = (15, 6), (16, 6)$ to the results of \cite{MN11}. These results are collected in the next statement.
\begin{proposition} 
\label{prop:4_dis_equal_J}
{$A(\cJ^{n, w}, 4) = \binom{n-w+4}{4}$ if $n$ and $w$ satisfy one of the following conditions:
\begin{enumerate}
  \item $w=5$ and $15\le n\le 70$;
  \item $w=6$ and $15\le n\le70$;
  \item $w=7$ and $20\le n\le 70$;
  \item $w=8$ and $25\le n\le 70$;
  \item $w=9$ and $30\le n\le 70$;
  \item $w=10$ and $35\le n\le 40$ or $48\le n\le 70$;
  \item $w=11$ and $60\le n\le 70$.
\end{enumerate}}
\end{proposition}
For $w=6$ these results augment our LP proof in Theorem~\ref{thm:Johnson4} to 
$n$ in the segment $15\le n\le 29$, and for $n=7$ they contribute the values $20\le n\le 44.$ For larger $w$ these results remain the best known.

In solving LP problems we relied on Mathematica, while for SDP problems we
used MATLAB with CVX toolbox \cite{GB08},\cite{GB14} and the MOSEK solver. All of our source codes can be found on github:
\begin{center}
    \url{https://github.com/PinChiehTseng/s-distance-set}.
\end{center}

\section*{Acknowledgement}
We would like to thank Hiroshi Nozaki for helpful discussions.

AB was supported by NSF grant CCF2104489 and NSF-BSF grant CCF2110113. AG was supported by NSF grant DMS-2054536. PCT and CYL were supported by the Ministry of Science and Technology (MOST) in Taiwan under Grant MOST110-2628-E-A49-007, MOST111-2119-M-A49-004, and MOST111-2119-M-001-002. WHY was supported by MOST under Grant109-2628-M-008-002-MY4.

\input{output.bbl}

\eject
\setcounter{page}{1}
\newgeometry{top=1.2in,bottom=1.2in,left=1.5in,right=1.5in}
{\small\appendix
\section{}

\renewcommand{\thefootnote}{\arabic{footnote}}
	\setcounter{footnote}{0}		
		{\renewcommand{\thefootnote}{}\footnotetext{
			\vspace*{-.15in}

			\rmark{\bf Note:} The material in Appendices A and B serves as supplementary data for the main text of the paper. It will appear in the preprint
version of this manuscript, but is not intended to be a part of the published journal paper.}}
	\renewcommand{\thefootnote}{\arabic{footnote}}

\vspace*{,1in}
In this appendix we list the expressions that support proofs of the statements of Theorems \ref{thm:Johnson2} to \ref{thm:Johnson4}.

\subsection{The case \texorpdfstring{$w=4, s=2$}.}\label{App:4,2}\hfill\\

\begin{enumerate}[(i)]
\item $\sD=\{1,2\}$

 \begin{gather*}
  E_1=\frac{(n-9)(n-2)}{4(n-4)(n-5)}, \;E_2=\frac{n-3}{(n-4)(n-5)}, \;
  E_3=\frac{n-6}{2(n-4)^2(n-5)},\\
  |C| \leq \frac{(n-2)(n-3)}{2} = \binom{n-2}{2}.
  \end{gather*}
This bound holds true for $n\ge 9$.

\item $\sD=\{2,3\}$

\begin{gather*}
E_1=\frac{(n-1)(n-11)}{6 (n-4) (n-5)}, E_2=\frac{(n-2)(2n-15)}{2 (n-4)(n-5)(n-6)}, E_3=\frac{3 n^2-41 n+142}{4 (n-4) (n-5)^2 (n-6)},\\
|C|\le \frac{(n-2) (n-1)\left(2 n^2-28 n+93\right)}{3\left(3 n^2-41 n+142\right)}\le \binom{n-2}{2}.
\end{gather*}
This bound holds true for $n\ge 11.$

\item $\sD=\{2,4\}$

\begin{gather*}
E_1=\frac{n}{2(n-4)}, E_2=\frac{(n-2)(n-12)}{(n-4)(n-5)(n-6)}, E_3=\frac{4(n-3)(n-8)}{ (n-6) (n-5) (n-4)^2},\\
|C|\le \frac{n(n-2)}{8}\le \binom{n-2}2.
\end{gather*}
This bound holds true for $n\ge 12.$

\item $\sD=\{3,4\}$

\begin{gather*}
E_1=\frac{n}{4(n-4)}, E_2=\frac{3(n-1)}{2(n-4)(n-5)}, E_3=\frac{3 (n-2)}{(n-4)^2 (n-5)},\\
|C|\le \frac{n(n-1)}{12}\le \binom{n-2}2.
\end{gather*}
This bound holds true for $n\ge 6.$

\end{enumerate}
  
\subsection{The case \texorpdfstring{$w=5,s=2$}.}  \label{App:5,2}\hfill\\

\begin{enumerate}[(i)]
\item $\sD=\{1,2\}$

\begin{gather*}
 E_1=\frac{(n-3)(n-12)}{5(n-5) (n-6)},E_2=\frac{n-4}{ (n-5)(n-6)},E_3=\frac{2 (n-8)}{5 (n-5)^2 (n-6)},\\
 |C|\le\frac{1}{2} (n-4) (n-3)=\binom{n-3}{2}.
\end{gather*}
This bound holds true for $n\ge 12$.

\item $\sD=\{2,3\}$

\begin{gather*}
E_1=\frac{(n-2)(n^2-25n+138)}{10 (n-5)(n-6) (n-7)}, E_2=\frac{4(n-3)(n-10)}{5(n-5)(n-6)(n-7)}, \\
E_3=\frac{3 \left(3 n^2-55 n+258\right)}{25 (n-5)^2(n-6) (n-7) },\\
|C|\le\frac{(n-3) (n-2) \left(5 n^2-95 n+408\right)}{6 \left(3 n^2-55 n+258\right)}\le\binom{n-3}2.
\end{gather*}
This bound holds true for $n\ge17$.

\item $\sD=\{2,4\}$

\begin{gather*}
E_1=\frac{(n-1)(3n-38)}{10 (n-5)(n-6)},E_2=\frac{2(n-3)(2n^2-45n+244)}{5(n-5)(n-6)(n-7)(n-8)},\\
E_3=\frac{4 \left(8 n^3-236 n^2+2265 n-7062\right)}{25(n-5)(n-6)^2(n-7) (n-8) },\\
|C|\le\frac{(n-3) (n-1) \left(15 n^3-420 n^2+3764 n-10832\right)}{8 \left(8 n^3-236 n^2+2265 n-7062\right)}\le\binom{n-3}2.
\end{gather*}
This bound holds true for $n\ge 14.$

\item $\sD=\{3,4\}$

\begin{gather*}
E_1=\frac{(n-1) (n-18)}{10(n-5) (n-6)},E_2=\frac{9(n-2)(n-11) }{10 (n-5)(n-6) (n-7)},\\
E_3=\frac{6 (n-4) \left(3 n^2-59 n+306\right)}{25 (n-7) (n-6)^2 (n-5)^2},\\
|C|\le\frac{(n-2) (n-1) \left(5 n^2-105 n+486\right)}{12 \left(3 n^2-59 n+306\right)}\le\binom{n-3}2.
\end{gather*}
This bound holds true for $n\ge 18.$

\item $\sD=\{4,5\}$
\begin{gather*}
E_1=\frac{n}{5(n-5)},E_2=\frac{8(n-1)}{5 (n-5)(n-6)},E_3=\frac{4 (n-2)}{(n-5)^2(n-6) },\\
|C|\le\frac{1}{20} (n-1) n\le\binom{n-3}2.
\end{gather*}
This bound holds true for $n\ge 7.$

\end{enumerate}

\subsection{The case \texorpdfstring{$w=6,s=2$}.}\label{App:6,2} \hfill\\

\begin{enumerate}[(i)]
\item $\sD = \{1, 2\}$

\begin{gather*}
E_1=\frac{(n-4)(n-15)}{6(n-6) (n-7)},E_2=\frac{n-5}{ (n-6)(n-7)},E_3=\frac{n-10}{3 (n-6)^2(n-7)},\\
|C| \leq \frac{(n-4)(n-5)}{2} = \binom{n-4}{2}.
\end{gather*}
This bound holds true for $n\ge 15$.

\item $\sD = \{2, 3\}$

\begin{gather*}
E_1=\frac{(n-3)\left(n^2-33n+236\right)}{15(n-6) (n-7)(n-8)},E_2=\frac{(n-4)(2n-25)}{3 (n-6) (n-7)(n-8)},E_3=\frac{n^2-23 n+136}{5 (n-6)^2 (n-7)^2},\\
|C| \leq \frac{(n-3)(n-4)(n^2-24n+125)}{3(n^2-23n+136)}\leq \binom{n-4}{2}.
\end{gather*}
This bound holds true for $n\ge23.$

\item $\sD = \{2, 4\}$

\begin{gather*}
E_1=\frac{(n-2)\left(n^2-27 n+172\right)}{5 (n-6) (n-7) (n-8) },E_2=\frac{2(n-4)(n-11)(n-15)}{3(n-6)(n-7)(n-8)(n-9)},\\E_3=\frac{8 (n-12) \left(n^2-27 n+167\right)}{15 (n-6)^2(n-7) (n-8) (n-9) },\\
|C| \leq \frac{(n-2)(n-4)(3n^3-104n^2+1145n-4020)}{8(n-12)(n^2-27n+167)} \leq \binom{n-4}{2}.
\end{gather*}
This bound holds true for $n\ge35.$

\item $ \sD = \{3, 4\}$

\begin{gather*}
E_1=\frac{(n-2)(n-11)(n-28)}{20(n-6) (n-7) (n-8)},E_2=\frac{(n-3)\left(3n^2-75n+452\right)}{5 (n-6)(n-7) (n-8)(n-9)},\\E_3=\frac{2 \left(3 n^4-136 n^3+2334 n^2-17831 n+50916\right)}{25 (n-9) (n-8)^2 (n-7)^2 (n-6)},\\
|C| \leq \frac{(n-2)(n-3)(5n^4-240n^3+4019n^2-28528n+73488)}{8(3n^4-136n^3+2334n^2-17831n+50916)} \leq \binom{n-4}{2}.
\end{gather*}
This bound holds true for $n\ge28$.

\item $\sD = \{3, 6\}$

\begin{gather*}
E_1=\frac{n}{2(n-6)},E_2=\frac{3(n-3)\left(n^2-33 n+316\right)}{5 (n-6)(n-7)(n-8)(n-9)},\\
E_3=\frac{18 (n-4) \left(n^2-32 n+237\right)}{5(n-6)^2 (n-7) (n-8) (n-9) },\\
|C| \leq \frac{n(n-3)(5n^2-109n+636)}{36(n^2-32n+237)} \leq \binom{n-4}{2}.
\end{gather*}
This bound holds true for $n\ge 27.$

\item $\sD = \{4, 5\}$

\begin{gather*}
E_1=\frac{(n-1) (n-27)}{15(n-6) (n-7)},E_2=\frac{2(n-2)(2n-31) }{5 (n-6)(n-7)(n-8)}, E_3=\frac{2 (n-4) \left(n^2-27 n+200\right)}{3 (n-6)^2(n-7)^2(n-8)  },\\
|C| \leq \frac{(n-1)(n-2)(n^2-30n+181)}{10(n^2-27n+200)} \leq \binom{n-4}{2}.
\end{gather*}
This bound holds true for $n\ge27.$

\item $\sD = \{4, 6\}$

\begin{gather*}
E_1=\frac{n}{3(n-6)},E_2=\frac{4(n-2)(n-28)}{5 (n-6)(n-7)(n-8)}, E_3=\frac{8 (n-3) (3 n-56)}{5 (n-6)^2(n-7) (n-8) },\\
|C| \leq \frac{n(n-2)(5n-68)}{24(3n-56)} \leq \binom{n-4}{2}.
\end{gather*}
This bound holds true for $n\ge28.$

\item $\sD = \{5, 6\}$

\begin{gather*}
E_1=\frac{n}{6(n-6)},E_2=\frac{5(n-1)}{3 (n-6)(n-7)},E_3=\frac{5 (n-2)}{(n-6)^2(n-7) },\\
|C| \leq \frac{n(n-1)}{30} \leq \binom{n-4}{2}.
\end{gather*}
This bound holds true for $n\ge8.$

\end{enumerate}

\subsection{The case \texorpdfstring{$w=4,s=3$}.}  \label{App:4,3}\hfill\\
\begin{enumerate}[(i)]
\item $\sD=\{1,2,3\}$
\begin{gather*}
  E_1=-\frac{(n-1)(n-2)\left(n^2-14n+51\right)}{24(n-4)(n-5)^2(n-6)}, \;E_2=-\frac{(n-1)(n-3)(n-8)}{6(n-4)^2(n-5)(n-6)},\\
  \;E_3=-\frac{(n-2)(n-3)}{2(n-4)^2(n-5)^2},\;
  E_4=-\frac{1}{4(n-4)^{2}(n-5)}, \\
  |C| \leq \frac{(n-1)(n-2)(n-3)}{6} = \binom{n-1}{3}.
\end{gather*}
This bound holds true for $n\geq 8$.

\item $\sD=\{1,3,4\}$

\begin{gather*}
E_1=-\frac{n(n-1) (n-2)}{8(n-4)^2(n-5)}, E_2=-\frac{n(n-3) (n-11)}{4(n-4)(n-5)(n-6) (n-7)},\\ E_3=-\frac{3(n-1)(n-3)(n-8)}{2 (n-4)^2(n-5)(n-6)(n-7)},
E_4=-\frac{3(n-2)}{(n-4)^2(n-5)(n-7)},\\
|C|\le \frac{n(n-1)(n-3)(n-8)}{24(n-6)}\le \binom{n-1}{3}.
\end{gather*}
This bound holds true for $n\ge 11.$

\item $\sD=\{2,3,4\}$

\begin{gather*}
E_1=-\frac{n(n-1)(n-2)}{24(n-4)^2(n-5)}, E_2=-\frac{n(n-2) (n-3)}{4(n-4)^2(n-5) (n-6)},\\ E_3=-\frac{3(n-1)(n-2)}{4(n-4)(n-5)^2(n-6)},
E_4=-\frac{(n-2)(n-3)}{(n-4)^2(n-5)^2(n-6)},\\
|C|\le \frac{n(n-1)(n-2)}{24}\le \binom{n-1}{3}. 
\end{gather*}
This bound holds true for $n\ge 7.$

\end{enumerate}

\subsection{The case \texorpdfstring{$w=5,s=3$}.}  \label{App:5,3}\hfill\\
\begin{enumerate}[(i)]
\item $\sD=\{1,2,3\}$

\begin{gather*}
E_1=-\frac{(n-2)(n-3)\left(n^2-19n+96\right)}{50(n-5)^2 (n-6)(n-7)}, E_2=-\frac{(n-2)(n-4)(n-12)}{10(n-5)^2(n-6)^2},\\ E_3=-\frac{2(n-3)(n-4)(n-8)}{5(n-5)^2(n-6)^2(n-7)},
E_4=-\frac{3(n-8)}{25(n-5)^3(n-6)},\\
|C|\le \frac{(n-2)(n-3)(n-4)}{6}\le \binom{n-2}{3}.
\end{gather*}
This bound holds true for $n\ge 12.$

\item $\sD=\{1,3,4\}$

\begin{gather*}
E_1=-\frac{(n-1)(n-2)(n-4)(n-9)(n-12)}{25(n-5)^2 (n-6)^2 (n-7)}, E_2=-\frac{(n-1)(n-4)\left(n^2-29n+174\right)}{10(n-5)^2(n-6)(n-7)(n-8)},\\ E_3=-\frac{3(n-2)(n-4)\left(3n^2-55n+246\right)}{10(n-5)^2(n-6)^2(n-7)(n-8)},
E_4=-\frac{6(n-4)\left(3n^2-53n+244\right)}{25(n-5)^3(n-6)(n-7)(n-8)},\\
|C| \leq \frac{\left(n - 1\right) \left(n - 2\right) \left(n - 4\right) \left(2 n^{2} - 43 n + 219\right)}{12 \left(3 n^{2} - 53 n + 244\right)} \le \binom{n-2}{3}.
\end{gather*}
This bound holds true for $n\ge 21.$

\item $\sD=\{2,3,4\}$

\begin{gather*}
E_1=-\frac{(n-1)(n-2)(n-4)\left(n^2-21n+126\right)}{100(n-5)^2 (n-6)^2 (n-7)}, \\
E_2=-\frac{2(n-1)(n-3)(n-10)(n-12)}{25(n-5)(n-6)^2(n-7)(n-8)},\\ 
E_3=-\frac{3(n-2)(n-3)\left(3n^2-53n+236\right)}{25(n-5)^2(n-6)(n-7)^2(n-8)},\\
E_4=-\frac{6(n-4)\left(4n^3-99n^2+819n-2274\right)}{125(n-5)^2(n-6)^2(n-7)^2(n-8)},\\
|C|\le \frac{(n-1)(n-2)(n-3)\left(5n^3-125n^2+1050n-2904\right)}{24\left(4n^3-99n^2+819n-2274\right)}\le \binom{n-2}{3}.
\end{gather*}
This bound holds true for $n\ge 12.$

\item $\sD=\{2,3,5\}$

\begin{gather*}
E_1=-\frac{n(n-2)(n-3)(2n-35)}{50(n-5)^2 (n-6) (n-7)}, E_2=-\frac{2n(n-3)(n-4)(4n-47)}{25(n-5)^2(n-6)(n-7)(n-8)},\\ E_3=-\frac{3(n-2)(n-3)\left(3n^2-85n+508\right)}{25(n-5)^2(n-6)(n-7)^2(n-8)},
E_4=-\frac{3(n-3)(n-4)\left(3n^2-55n+254\right)}{5(n-5)^3(n-6)(n-7)^2(n-8)},\\
|C|\le \frac{n(n-2)(n-3)\left(2n^2-45n+229\right)}{30\left(3n^2-55n+254\right)}\le \binom{n-2}{3}.
\end{gather*}
This bound holds true for $n\ge 20.$

\item $\sD=\{3,4,5\}$

\begin{gather*}
E_1=-\frac{n(n-1)(n-2)}{50(n-5)^2 (n-6)}, E_2=-\frac{9n(n-2)(n-3)}{50(n-5)^2(n-6)(n-7)},\\ E_3=-\frac{18(n-1)(n-2)(n-4)}{25(n-5)^2(n-6)^2(n-7)},
E_4=-\frac{6(n-2)(n-3)(n-4)}{5(n-5)^3(n-6)^2(n-7)},\\
|C|\le \frac{n(n-1)(n-2)}{60}\le \binom{n-2}{3}.
\end{gather*}
This bound holds true for $n\ge 8.$

\end{enumerate}

\subsection{The case \texorpdfstring{$w=6,s=3$}.}  \label{App:6,3}\hfill\\
\begin{enumerate}[(i)]
\item $\sD=\{1,2,3\}$

\begin{gather*}
E_1=-\frac{(n-3)(n-4)\left(n^2-24n+155\right)}{90(n-6)^2 (n-7)^2}, E_2=-\frac{(n-3)(n-5)(n-9)(n-16)}{15(n-6)^2(n-7)^2(n-8)},\\ E_3=-\frac{(n-4)(n-5)(n-10)}{3(n-6)^2(n-7)^2(n-8)},
E_4=-\frac{(n-9)(n-10)}{15(n-6)^3(n-7)^2},\\
|C|\le \frac{(n-3)(n-4)(n-5)}{6}= \binom{n-3}{3}.
\end{gather*}
This bound holds true for $n\ge 16.$

\item $\sD=\{1,3,4\}$

\begin{gather*}
E_1=-\frac{(n-2)(n-3)\left(5n^4-240n^3+4187n^2-31704n+88304\right)}{300(n-6) (n-7)^2(n-8)^2(n-9)},\\ E_2=-\frac{(n-2)(n-5)\left(n^3-50n^2+685n-2832\right)}{20(n-6)^2(n-7)^2(n-8)(n-9)},\\ E_3=-\frac{(n-3)(n-5)\left(3n^2-68n+368\right)}{5(n-6)^2(n-7)^2(n-8)^2},
E_4=-\frac{2\left(3n^2-70n+433\right)}{25(n-6)^2(n-7)^2(n-8)},\\
|C|\le \frac{(n-2)(n-3)(n-5)\left(5n^2-145n+932\right)}{24\left(3n^2-70n+433\right)} \leq \binom{n-3}{3}.
\end{gather*}
This bound holds true for $n\ge 31.$

\item $\sD=\{2,3,4\}$

\begin{gather*}
E_1=-\frac{(n-2)(n-3)\left(n^4-48n^3+895n^2-7464n+23152\right)}{300(n-6) (n-7)^2(n-8)^2(n-9)},\\ E_2=-\frac{(n-2)(n-4)(n-12)\left(n^2-28n+175\right)}{30(n-6)^2(n-7)(n-8)^2(n-9)},\\ 
E_3=-\frac{(n-3)(n-4)\left(3n^2-67n+380\right)}{15(n-6)^2(n-7)^2(n-8)(n-9)},
E_4=-\frac{2(n-12)\left(2n^2-39n+199\right)}{75(n-6)^2(n-7)^2(n-8)(n-9)},\\
|C|\le \frac{(n-2)(n-3)(n-4)\left(n^3-32n^2+347n-1220\right)}{8(n-12)\left(2n^2-39n+199\right)} \leq \binom{n-3}{3}.
\end{gather*}
This bound holds true for $n\ge 19.$

\item $\sD=\{2,3,5\}$

\begin{gather*}
E_1=-\frac{(n-1)(n-3)(n-5)\left(3n^3-140n^2+2081n-9876\right)}{225(n-6)^2(n-7)^2(n-8)(n-9)},\\ E_2=-\frac{(n-1)(n-4)\left(6n^3-238n^2+3057n-12785\right)}{45(n-6)(n-7)^2(n-8)(n-9)(n-10)},\\ E_3=-\frac{(n-3)(n-4)\left(n^3-43n^2+586n-2580\right)}{5(n-6)^2(n-7)^2(n-8)(n-9)(n-10)},\\
E_4=-\frac{(n-5)\left(5n^4-230n^3+3909n^2-29144n+80432\right)}{15(n-6)^2(n-7)^3(n-8)(n-9)(n-10)},\\
|C|\le \frac{(n-1)(n-3)(n-4)\left(3n^4-149n^3+2660n^2-20293n+55840\right)}{15\left(5n^4-230n^3+3909n^2-29144n+80432\right)} \leq \binom{n-3}{3}.
\end{gather*}
This bound holds true for $n\ge 22.$

\item $\sD=\{2,4,6\}$

\begin{gather*}
E_1=-\frac{n(n-2)(n-3)(n-20)}{15(n-6)^2(n-7)(n-8)}, \\
E_2=-\frac{2n(n-4)(n-5)\left(n^2-31n+246\right)}{9(n-6)^2(n-7)(n-8)(n-9)(n-10)},\\ 
E_3=-\frac{8(n-2)(n-4)\left(n^3-51n^2+809n-4020\right)}{15(n-6)^2(n-7)(n-8)^2(n-9)(n-10)},\\
E_4=-\frac{16(n-3)(n-5)(n-12)\left(3n^2-86n+536\right)}{15(n-6)^3(n-7)(n-8)^2(n-9)(n-10)},\\
|C|\le \frac{n(n-2)(n-4)}{48} \leq \binom{n-3}{3}.
\end{gather*}
This bound holds true for $n\ge 26.$

\item $\sD=\{3,4,5\}$

\begin{gather*}
E_1=-\frac{(n-1)(n-2)(n-4)\left(n^2-30n+281\right)}{300(n-6)^2(n-7)^2(n-8)},\\ 
E_2=-\frac{(n-1)(n-3)(n-5)\left(3n^2-93n+724\right)}{75(n-6)^2(n-7)^2(n-8)(n-9)},\\ 
E_3=-\frac{2(n-2)(n-3)\left(3n^2-71n+436\right)}{25(n-6)(n-7)^2(n-8)^2(n-9)},\\
E_4=-\frac{2(n-4)(n-5)(n-13)\left(n^2-20n+111\right)}{15(n-6)^2(n-7)^3(n-8)^2(n-9)},\\
|C|\le \frac{(n-1)(n-2)(n-3)\left(n^3-34n^2+401n-1524\right)}{40(n-13)\left(n^2-20n+111\right)} \leq \binom{n-3}{3}.
\end{gather*}
This bound holds true for $n\ge 15.$

\item $\sD=\{3,4,6\}$

\begin{gather*}
E_1=-\frac{n(n-2)(n-3)(n-28)}{60(n-6)^2(n-7)(n-8)}, E_2=-\frac{n(n-3)(n-4)(n-17)}{5(n-6)^2(n-7)(n-8)(n-9)},\\ 
E_3=-\frac{2(n-2)(n-3)(n-28)(3n-37)}{25(n-6)(n-7)^2(n-8)^2(n-9)},E_4=-\frac{6(n-3)(n-4)\left(6n^2-155n+1036\right)}{25(n-6)^2(n-7)^2(n-8)^2(n-9)},\\
|C|\le \frac{n(n-2)(n-3)\left(5n^2-165n+1132\right)}{72\left(6n^2-155n+1036\right)} \leq \binom{n-3}{3}.
\end{gather*}
This bound holds true for $n\ge 28.$

\item $\sD=\{3,5,6\}$

\begin{gather*}
E_1=-\frac{n(n-1)(n-2)}{30(n-6)^2(n-7)},E_2=-\frac{n(n-3)(n-4)(n-29)}{10(n-6)^2(n-7)(n-8)(n-9)},\\ 
E_3=-\frac{(n-1)(n-3)(n-5)(n-20)}{(n-6)^2(n-7)^2(n-8)(n-9)}, E_4=-\frac{3(n-2)(n-4)(n-5)(n-15)}{(n-6)^3(n-7)^2(n-8)(n-9)},\\
|C|\le \frac{n(n-1)(n-3)}{90} \leq \binom{n-3}{3}.
\end{gather*}
This bound holds true for $n\ge 29.$

\item $\sD=\{4,5,6\}$

\begin{gather*}
E_1=\frac{n(n-1)(n-2)}{90(n-6)^2(n-7)}, E_2=\frac{2n(n-2)(n-3)}{15(n-6)^2(n-7)(n-8)},\\ E_3=\frac{2(n-1)(n-2)(n-4)}{3(n-6)^2(n-7)^2(n-8)},
E_4=-\frac{4(n-2)(n-3)(n-4)}{3(n-6)^3(n-7)^2(n-8)},\\
|C|\le \frac{n(n-1)(n-2)}{120} \leq \binom{n-3}{3}.
\end{gather*}
This bound holds true for $n\ge 9.$

\end{enumerate}

\subsection{The case \texorpdfstring{$w=7,s=3$}.}  \label{App:7,3}\hfill\\

\begin{enumerate}[(i)]
\item $\sD=\{1,2,3\}$

\begin{gather*}
E_1=\frac{(n-4)(n-5)(n-10)\left(n^2-29n+228\right)}{147(n-7)^2(n-8)^2(n-9)}, E_2=\frac{(n-4)(n-6)(n-11)(n-20)}{21(n-7)^2(n-8)^2(n-9)},\\ 
E_3=\frac{2(n-5)(n-6)(n-12)}{7(n-7)^2(n-8)^2(n-9)},
E_4=-\frac{2(n-10)(n-11)(n-12)}{49(n-7)^3(n-8)^2(n-9)},\\
|C|\le \frac{(n-4)(n-5)(n-6)}{6} = \binom{n-4}{3}.
\end{gather*}
This bound holds true for $n\ge 20.$

\item $\sD=\{1,3,4\}$

\begin{gather*}
E_1=\frac{2(n-3)(n-4)\left(n^3-50n^2+789n-4020\right)}{245(n-7)^2(n-8)(n-9)^2},\\ E_2=\frac{(n-3)(n-6)\left(n^3-64n^2+1075n-5340\right)}{35(n-7)^2(n-8)^2(n-9)^2},\\ 
E_3=\frac{3(n-4)(n-6)(n-11)(n-17)}{7(n-7)^2(n-8)^2(n-9)^2},
E_4=-\frac{24(n-11)\left(n^2-29n+225\right)}{245(n-7)^3(n-8)(n-9)^2},\\
|C|\le \frac{(n-3)(n-4)(n-6)\left(2n^2-73n+555\right)}{24\left(n^2-29n+225\right)} \leq \binom{n-4}{3}.
\end{gather*}
This bound holds true for $n\ge 41.$

\item $\sD=\{2,3,4\}$

\begin{gather*}
E_1=\frac{(n-3)(n-4)\left(n^4-60n^3+1409n^2-14790n+57600\right)}{735(n-7)^2(n-8)(n-9)^2(n-10)},\\ E_2=\frac{4(n-3)(n-5)(n-12)(n-14)(n-25)}{245(n-7)^2(n-8)^2(n-9)(n-10)},\\ 
E_3=\frac{6(n-4)(n-5)\left(n^2-27n+186\right)}{49(n-7)^2(n-8)^2(n-9)^2},
E_4=-\frac{8\left(4n^3-153n^2+1973n-8640\right)}{1715(n-7)^3(n-8)^2(n-9)},\\
|C|\le \frac{(n-3)(n-4)(n-5)\left(7n^3-273n^2+3626n-15360\right)}{24\left(4n^3-153n^2+1973n-8640\right)} \leq \binom{n-4}{3}.
\end{gather*}
This bound holds true for $n\ge 25.$

\item $\sD=\{2,3,5\}$

\begin{gather*}
E_1=\frac{(n-2)(n-4)\left(4n^5-319n^4+9922n^3-150411n^2+1112664n-3218220\right)}{735(n-7)(n-8)^2(n-9)^2(n-10)(n-11)},\\ E_2=\frac{4(n-2)(n-5)\left(4n^4-238n^3+5169n^2-48717n+168372\right)}{245(n-7)^2(n-8)(n-9)^2(n-10)(n-11)},\\ 
E_3=\frac{2(n-4)(n-5)\left(3n^3-146n^2+2305n-11922\right)}{49(n-7)^2(n-8)^2(n-9)^2(n-11)},\\
E_4=-\frac{2\left(15n^4-880n^3+18907n^2-177514n+616824\right)}{343(n-7)^2(n-8)^2(n-9)^2(n-11)},\\
\begin{aligned}
|C| &\le \frac{(n-2)(n-4)(n-5)\left(28n^4-1673n^3+36204n^2-335277 n+1114398\right)}{30\left(15n^4-880n^3+18907n^2-177514n+616824\right)}\\
&\leq \binom{n-4}{3}.
\end{aligned}
\end{gather*}
This bound holds true for $n\ge 23.$

\item $\sD=\{2,4,6\}$

\begin{gather*}
E_1=\frac{(n-1)(n-3)(n-5)\left(15n^3-860n^2+15396n-87184\right)}{735(n-7)^2(n-8)^2(n-9)(n-10)},\\ E_2=\frac{2(n-1)(n-5)\left(2n^4-143n^3+3664n^2-40228n+160800\right)}{49(n-7)(n-8)^2(n-9)(n-10)(n-11)(n-12)},\\ 
E_3=\frac{8(n-3)(n-5)\left(8n^5-660n^4+21005n^3-324540n^2+2446172n-7221360\right)}{245(n-7)^2(n-8)^2(n-9)(n-10)^2(n-11)(n-12)},\\
E_4=-\frac{48(n-5)\left(2n^2-57n+398\right)Q}{1715(n-7)^2(n-8)^3(n-9)(n-10)^2(n-11)(n-12)},\\
\begin{aligned}
|C| &\le \frac{(n-1)(n-3)(n-5)P}{144\left(2n^2-57n+398\right)Q}
\leq \binom{n-4}{3},
\end{aligned}
\end{gather*}
where
\begin{align*}
P &= 105n^6-9380n^5+335496n^4-6191744n^3+62472912n^2-327751616n\\
& \quad +699935232,\\
Q &= 8n^4-504n^3+10997n^2-100866n+332240.
\end{align*}
This bound holds true for $n\ge 29.$

\item $\sD=\{3,4,5\}$

\begin{gather*}
E_1=\frac{(n-2)(n-3)(n-6)\left(n^4-66n^3+1745n^2-20448n+87840\right)}{1225(n-7)^2(n-8)^2(n-9)^2(n-10)},\\ E_2=\frac{3(n-2)(n-4)\left(n^4-68n^3+1669n^2-17642n+68160\right)}{245(n-7)(n-8)^2(n-9)^2(n-10)(n-11)},\\ 
E_3=\frac{8(n-3)(n-4)\left(3n^4-160n^3+3222n^2-28895n+97050\right)}{245(n-7)^2(n-8)(n-9)^2(n-10)^2(n-11)},\\
E_4=-\frac{8(n-6)Q}{1715(n-7)^2(n-8)^2(n-9)^3(n-10)^2(n-11)},\\
\begin{aligned}
|C| &\le \frac{(n-2)(n-3)(n-4)P}{40Q}
\leq \binom{n-4}{3},
\end{aligned}
\end{gather*}
where
\begin{align*}
P &= 7n^6-539n^5+17507n^4-302361n^3+2903826n^2-14643760n +30251400,\\
Q &= 5n^6-375n^5+11786n^4-198771n^3+1896155 n^2-9689568n+20692440.
\end{align*}
This bound holds true for $n\ge 27.$

\item $\sD=\{3,4,6\}$

\begin{gather*}
E_1=\frac{(n-1)(n-3)(n-5)(n-14)(n-20)(n-34)}{245(n-7)^2(n-8)^2(n-9)(n-10)},\\ 
E_2=\frac{3(n-1)(n-4)(n-6)(n-17)\left(n^2-37n+310\right)}{49(n-7)^2(n-8)^2(n-9)(n-10)(n-11)},\\ 
E_3=\frac{24(n-3)(n-4)\left(n^4-70n^3+1749n^2-18740n+73100\right)}{245(n-7)^2(n-8)(n-9)^2(n-10)^2(n-11)},\\
E_4=-\frac{72(n-5)(n-6)Q}{1715(n-7)^3(n-8)^2(n-9)^2(n-10)^2(n-11)},\\
\begin{aligned}
|C| &\le \frac{(n-1)(n-3)(n-4)P}{72Q}
\leq \binom{n-4}{3},
\end{aligned}
\end{gather*}
where
\begin{align*}
P &= 7n^5-560n^4+17017n^3-246980n^2+1720516n-4621280,\\
Q &= 4n^5-291n^4+8382n^3-119717n^2+848126n-2384080.
\end{align*}
This bound holds true for $n\ge 34.$

\item $\sD=\{3,4,7\}$

\begin{gather*}
E_1=\frac{n(n-3)(n-4)\left(3n^2-161n+1930\right)}{245(n-7)^2(n-8)(n-9)(n-10)},\\ 
E_2=\frac{3n(n-4)(n-5)\left(3n^2-103n+910\right)}{49(n-7)^2(n-8)(n-9)(n-10)(n-11)},\\ 
E_3=\frac{24(n-3)(n-4)\left(n^4-70n^3+1974n^2-24815n+112850\right)}{245(n-7)^2(n-8)(n-9)^2(n-10)^2(n-11)},\\
E_4=-\frac{24(n-4)(n-5)\left(n^4-68n^3+1664n^2-17557n+67710\right)}{35(n-7)^3(n-8)(n-9)^2(n-10)^2(n-11)},\\
\begin{aligned}
|C| &\le \frac{n(n-3)(n-4)\left(3n^4-200n^3+4773n^2-48316n+176580\right)}{168\left(n^4-68n^3+1664n^2-17557n+67710\right)} \leq \binom{n-4}{3}.
\end{aligned}
\end{gather*}
This bound holds true for $n\ge 36.$

\item $\sD=\{3,5,6\}$

\begin{gather*}
E_1=\frac{2(n-1)(n-2)(n-4)\left(2n^2-82n+903\right)}{735(n-7)^2(n-8)^2(n-9)},\\ 
E_2=\frac{(n-1)(n-4)(n-6)(n-14)\left(n^2-68n+915\right)}{49(n-7)^2(n-8)^2(n-9)(n-10)(n-11)},\\ 
E_3=\frac{5(n-2)(n-4)(n-14)\left(3n^2-118n+1083\right)}{49(n-7)(n-8)^2(n-9)^2(n-10)(n-11)},\\
E_4=-\frac{6(n-4)(n-6)\left(15n^4-910n^3+20839n^2-212584n+809256\right)}{343(n-7)^2(n-8)^3(n-9)^2(n-10)(n-11)},\\
\begin{aligned}
|C| &\le \frac{(n-1)(n-2)(n-4)\left(28n^4-1827n^3+43176n^2-432251n+1535730\right)}{90\left(15n^4-910n^3+20839n^2-212584n+809256\right)}\\
& \leq \binom{n-4}{3}.
\end{aligned}
\end{gather*}
This bound holds true for $n\ge 50.$

\item $\sD=\{4,5,6\}$

\begin{gather*}
E_1=\frac{(n-1)(n-2)(n-4)\left(n^2-41n+564\right)}{735(n-7)^2(n-8)^2(n-9)},\\ 
E_2=\frac{16(n-1)(n-3)(n-5)\left(n^2-42n+452\right)}{735(n-7)^2(n-8)^2(n-9)(n-10)},\\ 
E_3=\frac{8(n-2)(n-3)(n-6)\left(n^2-31n+258\right)}{49(n-7)^2(n-8)^2(n-9)^2(n-10)},\\
E_4=-\frac{8(n-4)(n-5)(n-6)\left(4n^3-171n^2+2543n-13548\right)}{343(n-7)^3(n-8)^3(n-9)^2(n-10)},\\
\begin{aligned}
|C| &\le \frac{(n-1)(n-2)(n-3)\left(7n^3-315n^2+5096n-25392\right)}{120\left(4n^3-171n^2+2543n-13548\right)} \leq \binom{n-4}{3}.
\end{aligned}
\end{gather*}
This bound holds true for $n\ge 20.$

\item $\sD=\{4,6,7\}$

\begin{gather*}
E_1=\frac{n(n-1)(n-2)}{49(n-7)^2(n-8)},
E_2=\frac{8n(n-3)(n-4)(2n-85)}{245(n-7)^2(n-8)(n-9)(n-10)},\\ 
E_3=\frac{192(n-1)(n-3)(n-5)(n-29)}{245(n-7)^2(n-8)^2(n-9)(n-10)},
E_4=-\frac{24(n-2)(n-4)(n-5)(4n-87)}{35(n-7)^3(n-8)^2(n-9)(n-10)},\\
\begin{aligned}
|C| &\le \frac{n(n-1)(n-3)(5n-94)}{168(4n-87)} \leq \binom{n-4}{3}.
\end{aligned}
\end{gather*}
This bound holds true for $n\ge 43.$

\item $\sD=\{5,6,7\}$

\begin{gather*}
E_1=\frac{n(n-1)(n-2)}{147(n-7)^2(n-8)},
E_2=\frac{5n(n-2)(n-3)}{49(n-7)^2(n-8)(n-9)},\\ 
E_3=\frac{30(n-1)(n-2)(n-4)}{49(n-7)^2(n-8)^2(n-9)},
E_4=-\frac{10(n-2)(n-3)(n-4)}{7(n-7)^3(n-8)^2(n-9)},\\
\begin{aligned}
|C| &\le \frac{n(n-1)(n-2)}{210} \leq \binom{n-4}{3}.
\end{aligned}
\end{gather*}
This bound holds true for $n\ge 10.$

\end{enumerate}

\subsection{The case \texorpdfstring{$w=5,s=4$}.}  \label{App:5,4}\hfill\\
\begin{enumerate}[(i)]
\item $\sD=\{1,2,3,4\}$

\begin{gather*}
E_1=-\frac{(n-1)(n-2)(n-3)(n-4)(n-10)\left(n^2-15n+60\right)}{500(n-5)^2(n-6)^2(n-7)^2(n-8)},\\ E_2=-\frac{(n-1)(n-2)(n-4)^2\left(n^2-17n+78\right)}{100(n-5)^3(n-6)^2(n-7)(n-8)},\\ 
E_3=-\frac{(n-1)(n-3)(n-4)(n-10)}{25(n-5)^2(n-6)^2(n-7)^2},
E_4=-\frac{3(n-2)(n-3)(n-4)}{25(n-5)^3(n-6)^2(n-7)},\\
E_5=-\frac{6(n-4)}{125(n-5)^3(n-6)(n-7)},\\
|C|\le \frac{(n-1)(n-2)(n-3)(n-4)}{24}= \binom{n-1}{4}.
\end{gather*}
This bound holds true for $n\ge 10.$

\item $\sD=\{2,3,4,5\}$

\begin{gather*}
E_1=-\frac{n(n-1)(n-2)^2(n-3)(n-4)}{500(n-5)^3(n-6)^2(n-7)}, E_2=-\frac{2n(n-1)(n-2)(n-3)(n-4)}{125(n-5)^2(n-6)^2(n-7)(n-8)},\\ 
E_3=-\frac{9n(n-2)(n-3)^2(n-4)}{125(n-5)^3(n-6)(n-7)^2(n-8)},\\
E_4=-\frac{24(n-1)(n-2)(n-3)(n-4)}{125(n-5)^2(n-6)^2(n-7)^2(n-8)},\\
E_5=-\frac{6(n-2)(n-3)(n-4)^2}{25(n-5)^3(n-6)^2(n-7)^2(n-8)},\\
|C|\le \frac{n(n-1)(n-2)(n-3)}{120}\leq \binom{n-1}{4}.
\end{gather*}
This bound holds true for $n\ge 9.$

\end{enumerate}

\subsection{The case \texorpdfstring{$w=6,s=4$}.}  \label{App:6,4}\hfill\\
\begin{enumerate}[(i)]
\item $\sD=\{1,2,3,4\}$

\begin{gather*}
E_1=-\frac{(n-2)(n-3)(n-4)(n-13)\left(n^2-19n+100\right)}{1800(n-6)^2(n-7)^2(n-8)(n-9)},\\ E_2=-\frac{(n-2)(n-3)(n-5)\left(n^2-23n+148\right)}{300(n-6)^2(n-7)^3(n-8)},\\ 
E_3=-\frac{(n-2)(n-4)(n-5)(n-10)(n-15)}{60(n-6)^3(n-7)^2(n-8)(n-9)},
E_4=-\frac{(n-3)(n-4)(n-5)(n-10)}{15(n-6)^3(n-7)^3(n-8)},\\
E_5=-\frac{n-10}{75(n-6)^3(n-7)^2},\\
|C|\le \frac{(n-2)(n-3)(n-4)(n-5)}{24}= \binom{n-2}{4}.
\end{gather*}
This bound holds true for $n\ge 15.$

\item $\sD = \{1,4,5,6\}$

\begin{gather*}
E_1=-\frac{n(n-1)(n-2)^2(n-3)(n-4)}{180(n-6)^3(n-7)^2(n-8)},\\ E_2=-\frac{n(n-1)(n-2)(n-5)\left(n^2-27n+212\right)}{90(n-6)^2(n-7)(n-8)(n-9)(n-10)(n-11)},\\ 
E_3=-\frac{2n(n-2)(n-3)(n-5)\left(n^2-26n+178\right)}{15(n-6)^2(n-7)^2(n-8)(n-9)(n-10)(n-11)},\\
E_4=-\frac{2(n-1)(n-2)(n-4)(n-5)(n-15)}{3(n-6)^3(n-7)^2(n-8)(n-9)(n-11)},\\
E_5=-\frac{4(n-2)(n-3)(n-4)\left(n^2-24n+125\right)}{3(n-6)^3(n-7)^2(n-8)(n-9)(n-10)(n-11)},\\
|C|\le \frac{n(n-1)(n-2)(n-5)\left(n^2-23n+136\right)}{240\left(n^2-24n+125\right)} \leq \binom{n-2}{4}.
\end{gather*}
This bound holds true for $n\ge 17.$

\item $\sD=\{2,3,4,5\}$

\begin{gather*}
E_1=-\frac{(n-1)(n-2)(n-3)(n-4)(n-5)\left(n^3-34n^2+401n-1724\right)}{4500(n-6)^2(n-7)^3(n-8)^2(n-9)}, \\
E_2=-\frac{(n-1)(n-2)(n-4)^2\left(n^3-35n^2+431n-1855\right)}{450(n-6)^2(n-7)^2(n-8)^2(n-9)(n-10)},\\ 
E_3=-\frac{(n-1)(n-3)(n-4)(n-5)\left(3n^3-108n^2+1289n-5120\right)}{225(n-6)^2(n-7)^3(n-8)(n-9)^2(n-10)},\\
E_4=-\frac{(n-2)(n-3)(n-4)\left(4n^3-123n^2+1263n-4340\right)}{75(n-6)^2(n-7)^2(n-8)^2(n-9)^2(n-10)},\\
E_5=-\frac{(n-4)(n-5)\left(5n^4-194n^3+2821n^2-18244n+44376\right)}{225(n-6)^2(n-7)^3(n-8)^2(n-9)^2(n-10)},\\
|C|\le \frac{(n-1)(n-2)(n-3)(n-4)\left(n^4-39n^3+571n^2-3729n+9120\right)}{20\left(5n^4-194n^3+2821n^2-18244n+44376\right)} \leq \binom{n-2}{4}.
\end{gather*}
This bound holds true for $n\ge 15.$

\item $\sD=\{2,3,4,6\}$

\begin{gather*}
E_1=-\frac{n(n-2)(n-3)^2(n-4)\left(n^2-33n+284\right)}{900(n-6)^2(n-7)^2(n-8)^2(n-9)}, \\
E_2=-\frac{n(n-2)(n-3)(n-4)(n-5)(n-12)(n-22)}{90(n-6)^3(n-7)(n-8)^2(n-9)(n-10)},\\ 
E_3=-\frac{n(n-3)(n-4)^2(n-5)\left(n^2-27n+174\right)}{15(n-6)^3(n-7)^2(n-8)(n-9)^2(n-10)},\\
E_4=-\frac{2(n-2)(n-3)(n-4)\left(2n^3-99n^2+1369n-5820\right)}{75(n-6)^2(n-7)^2(n-8)^2(n-9)^2(n-10)},\\
E_5=-\frac{4(n-3)(n-4)(n-5)(n-12)\left(2n^2-39n+194\right)}{25(n-6)^3(n-7)^2(n-8)^2(n-9)^2(n-10)},\\
|C|\le \frac{n(n-2)(n-3)(n-4)\left(n^2-27n+158\right)}{144\left(2n^2-39n+194\right)} \leq \binom{n-2}{4}.
\end{gather*}
This bound holds true for $n\ge 30.$

\item $\sD=\{2,4,5,6\}$

\begin{gather*}
E_1=-\frac{n(n-1)(n-2)^2(n-3)(n-4)}{450(n-6)^3(n-7)^2(n-8)},\\ E_2=-\frac{n(n-1)(n-2)(n-4)(n-5)(n-25)}{135(n-6)^2(n-7)^2(n-8)(n-9)(n-10)},\\ 
E_3=-\frac{4n(n-2)(n-3)(n-4)(n-5)(n-18)}{45(n-6)^3(n-7)(n-8)^2(n-9)(n-10)},\\
E_4=-\frac{4(n-1)(n-2)(n-4)^2(n-14)}{9(n-6)^2(n-7)^2(n-8)^2(n-9)(n-10)},\\
E_5=-\frac{8(n-2)(n-3)(n-4)(n-5)(5n-56)}{45(n-6)^3(n-7)^2(n-8)^2(n-9)(n-10)},\\
|C|\le \frac{n(n-1)(n-2)(n-4)(3n-44)}{240(5n-56)} \leq \binom{n-2}{4}.
\end{gather*}
This bound holds true for $n\ge 25.$

\item $\sD=\{3,4,5,6\}$

\begin{gather*}
E_1=-\frac{n(n-1)(n-2)^2(n-3)(n-4)}{1800(n-6)^3(n-7)^2(n-8)}, E_2=-\frac{n(n-1)(n-2)(n-3)(n-4)(n-5)}{150(n-6)^3(n-7)^2(n-8)(n-9)},\\ 
E_3=-\frac{n(n-2)(n-3)^2(n-4)}{25(n-6)^2(n-7)^2(n-8)^2(n-9)},\\ E_4=-\frac{2(n-1)(n-2)(n-3)(n-4)(n-5)}{15(n-6)^2(n-7)^3(n-8)^2(n-9)},\\
E_5=-\frac{(n-2)(n-3)(n-4)^2(n-5)}{5(n-6)^3(n-7)^3(n-8)^2(n-9)},\\
|C|\le \frac{n(n-1)(n-2)(n-3)}{360} \leq \binom{n-2}{4}.
\end{gather*}

This bound holds true for $n\ge 10.$
\end{enumerate}

\subsection{The case \texorpdfstring{$w=7,s=4$}.}  \label{App:7,4}\hfill\\
\begin{enumerate}[(i)]
\item $\sD=\{1,2,3,4\}$

\begin{gather*}
E_1=-\frac{(n-3)(n-4)(n-5)(n-16)\left(n^2-23n+150\right)}{5145(n-7)^3(n-8)^2(n-9)},\\ E_2=-\frac{(n-3)(n-4)(n-6)(n-11)\left(n^2-29n+240\right)}{735(n-7)^3(n-8)^2(n-9)^2},\\ 
E_3=-\frac{2(n-3)(n-5)(n-6)(n-12)(n-20)}{245(n-7)^3(n-8)^3(n-9)},\\
E_4=-\frac{2(n-4)(n-5)(n-6)(n-11)(n-12)}{49(n-7)^3(n-8)^3(n-9)^2},\\
E_5=-\frac{8(n-10)(n-11)(n-12)}{1715(n-7)^4(n-8)^2(n-9)},\\
|C|\le \frac{(n-3)(n-4)(n-5)(n-6)}{24}= \binom{n-3}{4}.
\end{gather*}
This bound holds true for $n\ge 20.$

\item $\sD=\{1,4,5,6\}$

\begin{gather*}
E_1=-\frac{(n-1)(n-2)(n-3)(n-4)(n-5)(n-6)\left(n^3-45n^2+728n-3856\right)}{1715(n-7)^3(n-8)^3(n-9)^2(n-10)},\\ E_2=-\frac{(n-1)(n-2)(n-4)(n-6)\left(n^4-68n^3+1871n^2-23068n+102720\right)}{735(n-7)^3(n-8)^2(n-9)(n-10)(n-11)(n-12)},\\ 
E_3=-\frac{8(n-1)(n-3)(n-5)(n-6)\left(2n^4-136n^3+3409n^2-36656n+142272\right)}{735(n-7)^3(n-8)^2(n-9)^2(n-10)(n-11)(n-12)},\\
E_4=-\frac{8(n-2)(n-3)(n-6)^2\left(n^3-45n^2+668n-3216\right)}{49(n-7)^3(n-8)^3(n-9)^2(n-10)(n-12)},\\
E_5=-\frac{8(n-4)(n-5)(n-6)\left(4n^4-231n^3+4976n^2-47841n+172692\right)}{343(n-7)^4(n-8)^2(n-9)^2(n-10)(n-11)(n-12)},\\
\begin{aligned}
|C| &\le \frac{(n-1)(n-2)(n-3)(n-6)\left(3n^4-176n^3+4029n^2-41296n+156816\right)}{120\left(4n^4-231n^3+4976n^2-47841n+172692\right)}\\
&\leq \binom{n-3}{4}.
\end{aligned}
\end{gather*}
This bound holds true for $n\ge 23.$

\item $\sD=\{2,3,4,5\}$

\begin{gather*}
E_1=-\frac{(n-2)(n-3)(n-4)(n-6)A}{25725(n-7)^2(n-8)^2(n-9)^3(n-10)^2(n-11)},\\ 
E_2=-\frac{4(n-2)(n-3)(n-5)(n-6)(n-14)\left(n^4-53n^3+1099n^2-10227n+35460\right)}{8575(n-7)^3(n-8)^2(n-9)^2(n-10)^2(n-11)},\\
E_3=-\frac{6(n-2)(n-4)(n-5)\left(n^4-58n^3+1229n^2-11352n+38700\right)}{1715(n-7)^2(n-8)^2(n-9)^3(n-10)(n-11)},\\
E_4=-\frac{16(n-3)(n-4)(n-5)(n-14)\left(2n^2-47n+285\right)}{1715(n-7)^3(n-8)^2(n-9)^2(n-10)(n-11)},\\
E_5=-\frac{8(n-6)\left(5n^4-238n^3+4261n^2-34076n+103080\right)}{12005(n-7)^3(n-8)^2(n-9)^2(n-10)(n-11)},\\
\begin{aligned}
|C| &\le \frac{(n-2)(n-3)(n-4)(n-5)\left(7n^4-336n^3+6083n^2-49434n+150120\right)}{120\left(5n^4-238n^3+4261n^2-34076n+103080\right)}\\
&\leq \binom{n-3}{4},
\end{aligned}
\end{gather*}
where
\begin{align*}
    A &=n^6-77n^5+2501n^4-44023n^3+442518n^2-2399880n+5461200.
\end{align*}
This bound holds true for $n\ge 22.$

\item $\sD=\{2,3,4,6\}$

\begin{gather*}
E_1=-\frac{(n-1)(n-3)(n-4)(n-5)(n-6)A}{5145(n-7)^3(n-8)^2(n-9)^2(n-10)^2(n-11)},\\ 
E_2=-\frac{4(n-1)(n-3)(n-5)^2(n-14)(n-16)\left(n^2-40n+332\right)}{1715(n-7)^2(n-8)^3(n-9)(n-10)^2(n-11)},\\
E_3=-\frac{2(n-1)(n-4)(n-5)(n-6)\left(3n^4-187n^3+4228n^2-41548n+150720\right)}{343(n-7)^2(n-8)^3(n-9)^2(n-10)(n-11)(n-12)},\\
E_4=-\frac{8(n-3)(n-4)(n-5)\left(4n^4-273n^3+6491n^2-65742n+243360\right)}{1715(n-7)^3(n-8)^2(n-9)^2(n-10)(n-11)(n-12)},\\
E_5=-\frac{48(n-5)(n-6)\left(8n^5-522n^4+13459n^3-171769n^2+1086792n-2729600\right)}{12005(n-7)^3(n-8)^3(n-9)^2(n-10)(n-11)(n-12)},\\
\begin{aligned}
|C| &\le \frac{(n-1)(n-3)(n-4)(n-5)\left(7n^{5}-511n^{4}+14364n^{3}-196196n^{2}+1306400n-3389952\right)}{144\left(8n^5-522n^4+13459n^3-171769n^2+1086792n-2729600\right)}\\
&\leq \binom{n-3}{4},
\end{aligned}
\end{gather*}
where
\begin{align*}
    A &=n^5-80n^4+2527n^3-39548n^2+306604n-941120.
\end{align*}
This bound holds true for $n\ge 31.$

\item $\sD=\{2,4,5,6\}$

\begin{gather*}
E_1=-\frac{(n-1)(n-2)(n-3)(n-4)(n-5)(n-6)\left(n^3-45n^2+728n-4176\right)}{5145(n-7)^3(n-8)^3(n-9)^2(n-10)},\\ 
E_2=-\frac{2(n-1)(n-2)(n-4)(n-5)\left(2n^4-157n^3+4581n^2-57798n+265032\right)}{5145(n-7)^2(n-8)^2(n-9)^2(n-10)(n-11)(n-12)},\\
E_3=-\frac{64(n-1)(n-3)(n-5)^2(n-14)(n-16)\left(n^2-41n+354\right)}{5145(n-7)^2(n-8)^3(n-9)(n-10)^2(n-11)(n-12)},\\
E_4=-\frac{8(n-2)(n-3)(n-5)(n-6)\left(4n^4-221n^3+4585n^2-42276n+145944\right)}{343(n-7)^3(n-8)^2(n-9)^2(n-10)^2(n-11)(n-12)},\\
E_5=-\frac{16(n-4)(n-5)(n-6)Q}{2401(n-7)^3(n-8)^3(n-9)^2(n-10)^2(n-11)(n-12)},\\
\begin{aligned}
|C| &\le \frac{(n-1)(n-2)(n-3)(n-5)P}{240Q} \leq \binom{n-3}{4},
\end{aligned}
\end{gather*}
where
\begin{align*}
    P &=7n^5-483n^4+13412n^3-185556n^2+1267296n-3387456,\\
    Q &=8n^5-502n^4+12673n^3-161081n^2+1030386n-2646072.
\end{align*}
This bound holds true for $n\ge 29.$

\item $\sD=\{3,4,5,6\}$

\begin{gather*}
E_1 = -\frac{\left(n - 1\right) \left(n - 2\right) \left(n - 3\right) \left(n - 4\right) \left(n - 5\right) \left(n - 6\right) \left(n^{3} - 45 n^{2} + 728 n - 4656\right)}{25725 \left(n - 7\right)^3 \left(n - 8\right)^{3} \left(n - 9\right)^{2} \left(n - 10\right)}, \\
E_2 = -\frac{\left(n - 1\right) \left(n - 2\right) \left(n - 4\right)^{2} \left(n - 6\right) \left(n^{3} - 46 n^{2} + 769 n - 4620\right)}{1715 \left(n - 7\right)^2 \left(n - 8\right)^3 \left(n - 9\right)^{2} \left(n - 10\right) \left(n - 11\right)}, \\
E_3 = -\frac{8 \left(n - 1\right) \left(n - 3\right) \left(n - 4\right) \left(n - 5\right) \left(n - 6\right) \left(n^{3} - 47 n^{2} + 737 n - 3910\right)}{1715 \left(n - 7\right)^3 \left(n - 8\right)^{2} \left(n - 9\right)^{2} \left(n - 10\right)^{2} \left(n - 11\right)}, \\
E_4 = -\frac{8 \left(n - 2\right) \left(n - 3\right) \left(n - 4\right) \left(n - 6\right) \left(n - 15\right) \left(n^{2} - 24 n + 155\right)}{343 \left(n - 7\right)^2 \left(n - 8\right)^{2} \left(n - 9\right)^{3} \left(n - 10\right)^{2} \left(n - 11\right)}, \\
E_5 = -\frac{24 \left(n - 4\right) \left(n - 5\right) \left(n - 6\right)^2 \left(5 n^{4} - 246 n^{3} + 4585 n^{2} - 38568 n + 124776\right)}{12005 \left(n - 7\right)^3 \left(n - 8\right)^{3} \left(n - 9\right)^{3} \left(n - 10\right)^{2} \left(n - 11\right)}, \\
|C| \leq \frac{\left(n - 1\right) \left(n - 2\right) \left(n - 3\right) \left(n - 4\right) \left(7 n^{4} - 350 n^{3} + 6671 n^{2} - 58072 n + 189960\right)}{360 \left(5 n^{4} - 246 n^{3} + 4585 n^{2} - 38568 n + 124776\right)} \leq \binom{n-3}{4}.
\end{gather*}
This bound holds true for $n \ge 21$.

\item $\sD=\{3,4,6,7\}$

\begin{gather*}
E_1=-\frac{n(n-1)(n-2)(n-3)(n-4)(n-5)(n-38)}{1715(n-7)^3(n-8)^2(n-9)(n-10)},\\ 
E_2=-\frac{3n(n-1)(n-2)(n-4)(n-5)(n-6)(n-23)}{343(n-7)^3(n-8)^2(n-9)(n-10)(n-11)},\\
E_3=-\frac{24n(n-3)(n-4)^2(n-5)\left(n^2-61n+730\right)}{1715(n-7)^3(n-8)(n-9)^2(n-10)^2(n-11)},\\
E_4=-\frac{72(n-1)(n-3)(n-4)(n-5)(n-6)\left(4n^2-161n+1530\right)}{1715(n-7)^3(n-8)^2(n-9)^2(n-10)^2(n-11)},\\
E_5=-\frac{72(n-2)(n-4)(n-5)^2(n-6)\left(2n^2-59n+459\right)}{245(n-7)^4(n-8)^2(n-9)^2(n-10)^2(n-11)},\\
\begin{aligned}
|C| &\le \frac{n(n-1)(n-3)(n-4)\left(n^2-45n+410\right)}{504\left(2n^2-59n+459\right)} \leq \binom{n-3}{4}.
\end{aligned}
\end{gather*}
This bound holds true for $n\ge 45.$

\item $\sD=\{4,5,6,7\}$

\begin{gather*}
E_1=-\frac{n(n-1)(n-2)^2(n-3)(n-4)}{5145(n-7)^3(n-8)^2(n-9)},
E_2=-\frac{16n(n-1)(n-2)(n-3)(n-4)(n-5)}{5145(n-7)^3(n-8)^2(n-9)(n-10)},\\
E_3=-\frac{8n(n-2)(n-3)^2(n-4)(n-6)}{343(n-7)^3(n-8)^2(n-9)^2(n-10)},\\
E_4=-\frac{32(n-1)(n-2)(n-3)(n-4)(n-5)(n-6)}{343(n-7)^3(n-8)^3(n-9)^2(n-10)},\\
E_5=-\frac{8(n-2)(n-3)(n-4)^2(n-5)(n-6)}{49(n-7)^4(n-8)^3(n-9)^2(n-10)},\\
\begin{aligned}
|C| &\le \frac{n(n-1)(n-2)(n-3)}{840} \leq \binom{n-3}{4}.
\end{aligned}
\end{gather*}
This bound holds true for $n\ge 11.$

\end{enumerate}

\section{SDP bounds for \texorpdfstring{$\cH_2^n$}. and \texorpdfstring{$\cJ^{n,w}$}.}
\label{sec:SDPbounds}
The general results of \cite{S05} are specialized here for $s$-codes  in $\cH_2^n$ or $\cJ^{n,w}$.
 
Consider an $s$-code $C$ in $\cH_2^n$ with distances $\{d_{1}, \dots, d_{s}\}$. Its triple distance distribution is defined as

\begin{align*}
x_{i, j}^{t} = \frac{1}{|C| \binom{n}{i-t, j-t, t}} |\{ (u, v, z) \in C^{3} &: 
d_H(u, v) = i, d_H(u, z) = j,\\ &\;\; d_H(v,z)=i+j-2t\}|,
\end{align*}
The variables $x_{i,j}^t$ characterize the distance relations among triples of codewords. 
Together with the fact the $C$ has only $s$ distances, we obtain the following form of the SDP bound.

\begin{theorem}
\label{thm:sdp_Ham}
Let $C \subset \cH_{2}^{n}$ be an $s$-code with distance set 
$\sD = \{d_{1}, \dots, d_{s}\}$. Suppose that $x_{i, j}^{t} \in [0, 1]$  for $i, j, t \in \{0, \dots, n\}$ are variables  satisfying  the following conditions:
\begin{enumerate}[i.]
    \item
    $x_{0, 0}^{0} = 1$;
    \item
    $0 \leq x_{i, j}^{t} \leq x_{i, 0}^{0}$ and $x_{i, 0}^{0} + x_{j, 0}^{0} \leq 1 + x_{i, j}^{t}$;
    \item
    $x_{i, j}^{t} = x_{i', j'}^{t'}$ if $(i, j, i+j-2t)$ is a permutation of $(i', j', i'+ j'-2t')$;
    \item
    $x_{i, j}^{t} = 0$ if one of $i$, $j$, and $i+j-2t$ is not in $\mathscr D$.
\end{enumerate}
For $k = 0, \dots, \lfloor \frac{n}{2} \rfloor$, define
\begin{equation}
B_{k} = \left( \sum_{t} \binom{n-2k}{i-k}^{-\frac{1}{2}} \binom{n-2k}{j-k}^{-\frac{1}{2}} \beta_{i, j, k}^{t} x_{i, j}^{t} \right)_{i, j = k}^{n-k}\label{eq:B_k_ham}
\end{equation}
where
\begin{equation*}
\beta_{i, j, k}^{t} = \sum_{u = 0}^{n} (-1)^{u-t} \binom{u}{t} \binom{n-2k}{u-k}\binom{n-k-u}{i-u}\binom{n-k-u}{j-u}.
\end{equation*}
Then
\begin{equation*}
    |C| \leq {\rm max} \Big\{ \sum_{i = 0}^{n} \binom{n}{i}x_{i, 0}^{0} : B_{k} \text{ is p.s-d. for all } k= 0, \dots, \Big\lfloor \frac{n}{2} \Big\rfloor\Big\}. 
\end{equation*}
\end{theorem}

Consider a constant weight $s$-code $C$ in $\cJ^{n,w}$ with distances $\{d_{1}, \dots, d_{s}\}$. 
Let $\chi^A$ be the indicator vector of a set $A \subset \{1. \dots, n\}.$
Define a \textit{semidistance} between two vectors $u = (u_{1}, \dots, u_{n}),$ $v = (v_{1}, \dots, v_{n}) \in \mathbb{F}_2^n$  as an $n$-dimensional vector
\begin{equation*}
\bar{d}(u, v) = \chi^{\text{supp}(u)\setminus\text{supp}(v)},
\end{equation*}
where $\text{supp}(u)=\{i: u_i=1\}$ is the support of $u$.
Then the triple distance distribution of $C$ is defined as
\begin{equation*}
x_{i, j}^{t, r} = \frac{1}{|C| \binom{w}{i-t, j-t, t}\binom{n-w}{i-r, j-r, r}}  \lambda_{i, j}^{t, r},
\end{equation*}
where
\begin{align*}
\lambda_{i, j}^{t, r} = |\{ (u, v, z) \in C^{3} : 
d_{J}(u, v) &= i, d_{J}(u, z) = j,\\
& \bar{d}(u, v) \cdot \bar{d}(u, z) = t, \bar{d}(v, u) \cdot \bar{d}(z, u) = r\}|.
\end{align*}

Next we formulate Schrijver's SDP bound for the Johnson space, taking account of the fact that $C$ is 
an $s$-code and of the ensuing additional constraints on the distance coefficients $x_{i,j}^{t,r}.$ 

\begin{theorem}
\label{thm:sdp_Joh}
Let $C \subset \cJ^{n, w}$ be an $s$-code with distance set $\sD=\{d_{1}, \dots, d_{s}\}$. Suppose that $x_{i, j}^{t, r} \in [0, 1]$  for $i, j, t, r \in \{0, \dots, {\rm min}\{w, n-w\}\}$ are variables satisfying the following conditions:
\begin{enumerate}[i.]
    \item
    $x_{0, 0}^{0, 0} = 1$;
    \item
    $0 \leq x_{i, j}^{t, r} \leq x_{i, 0}^{0, 0}$ and $x_{i, 0}^{0, 0} + x_{j, 0}^{0, 0} \leq 1 + x_{i, j}^{t, r}$;
    \item
    $x_{i, j}^{t, r} = x_{i', j'}^{t', r'}$ if $t-r = t'-r'$ and $(i, j, i+j-t-r)$ is a permutation of $(i', j', i'+ j'-t'-r')$;
    \item
    $x_{i, j}^{t, r} = 0$ if one of $i$, $j$, and $i+j-t-r$ is not in $\sD$.
\end{enumerate}
For $k = 0, \dots, \lfloor \frac{w}{2} \rfloor$, $l = 0, \dots, \lfloor \frac{n-w}{2} \rfloor$,   define
\begin{equation}
B_{k, l} = \left( \sum_{t} \alpha(n, i, j, k, l)\beta_{i, j, k}^{t, w}\beta_{i, j, l}^{r, n-w} x_{i, j}^{t, r} \right)_{(i, j) = (k, l)}^{(n-k, n-l)} \label{eq:B_kl}
\end{equation}
where
\begin{equation*}
\beta_{i, j, k}^{t, a} = \sum_{u = 0}^{a} (-1)^{u-t} \binom{u}{t} \binom{a-2k}{u-k}\binom{a-k-u}{i-u}\binom{a-k-u}{j-u},
\end{equation*}
for $a = w, n-w$, and
\begin{equation*}
\alpha(n, i, j, k, l) = \binom{n-2k}{i-k}^{-\frac{1}{2}} \binom{n-2k}{j-k}^{-\frac{1}{2}} \binom{n-2l}{i-l}^{-\frac{1}{2}} \binom{n-2l}{j-l}^{-\frac{1}{2}}.
\end{equation*}
Then 
\begin{equation*}
    |C| \leq {\rm max} \left\{ \sum_{i = 0}^{n} \binom{w}{i}\binom{n-w}{i}x_{i, 0}^{0, 0} : B_{k, l} \text{ is p.s.-d. for all } k, l\right\}. 
\end{equation*}
\end{theorem}

}

\end{document}

%% file: output.bbl
\newcommand{\etalchar}[1]{$^{#1}$}